\documentclass{article}

\usepackage{amsfonts}
\usepackage{amsmath}

\usepackage{bm}
\usepackage{amsthm}
\usepackage{psfrag}
\usepackage{auto-pst-pdf}
\usepackage{enumerate}
\usepackage{amssymb}
\usepackage{bbm}
\usepackage{color}
\usepackage{algorithm}
\usepackage{algpseudocode}
\usepackage{subfigure}
\usepackage{enumitem}
\usepackage{url} 
\def\ie{{\it i.e.}}
\def\eg{{\it e.g.}}
\algnewcommand\algorithmicinput{\textbf{Input:}}
\algnewcommand\INPUT{\item[\algorithmicinput]}
\algnewcommand\Offline{\item[\textbf{Offline-Phase:}]}
\algnewcommand\Online{\item[\textbf{Online-Phase:}]}
\def\nn{\nonumber}
\def\expec{\mathbbm{E}}
\def\real{\mathbbm{R}}

\DeclareMathOperator{\diag}{diag}

\def\defeq{\triangleq}

\newcommand{\st}{\textup{subject to}}
\DeclareMathOperator{\minimize}{\text{\textnormal{minimize}}}
\DeclareMathOperator{\rank}{rank}
\newtheorem{theorem}{Theorem}[section]
\newtheorem{lemma}[theorem]{Lemma}

\newtheorem{example}[theorem]{Example}

\newcommand{\vv}[1] {\mathbf{#1}}
\def\b{\vv{b}}
\def\u{\vv{u}}

\def\f{\vv{f}}

\def\p{\vv{p}}
\def\d{\vv{d}}
\def\h{\vv{h}}

\def\b{s}
\def\u{u}
\def\C{C_{K}}

\def\f{f}
\def\psiaux{\bm{\kappa}}

\def\gmax{\g^{\max}}
\def\bmax{S^{\max}}
\def\bmin{S^{\min}}
\def\bdot{S^{(\cdot)}}
\def\umax{U^{\max}}
\def\umin{U^{\min}}
\def\udot{U^{(\cdot)}}

\def\upos{\u^{+}}
\def\uneg{\u^{-}}
\def\p{p}

\def\d{\delta}
\def\la{\lambda}
\def\muC{\mu^\mathrm{C}}
\def\muD{\mu^\mathrm{D}}
\def\alI{\alpha^\delta}
\def\alC{\alpha^\mathrm{C}}
\def\alD{\alpha^\mathrm{D}}

\def\alF{\alpha^\mathrm{F}}
\def\alConst{\alpha^\mathrm{Const}}
\newcommand{\pos}[1]{\left(#1\right)^+}
\renewcommand{\neg}[1]{\left(#1\right)^-} 
\def\minimize{\mbox{minimize  }}
\def\st{\mbox{subject to  }}

\def\h{h}
\def\hC{\h^\mathrm{C}}
\def\hD{\h^\mathrm{D}}

\def\g{g}
\def\MWh{\text{MWh}}
\def\MW{\text{MW}}
\def\hr{\text{h}}
\newcommand{\opttag}[1]{\mbox{\bf#1  }}
\def\G{G}
\def\V{V}
\def\E{E}
\def\ER{\E^\mathrm{R}}
\def\fn{\vv{\f}}
\def\un{\vv{\u}}
\def\bn{\vv{\b}}
\def\bsn{\vv{\bs}}
\def\ksn{\vv{\ks}}
\def\ustatn{\vv{\ustat}}
\def\fstatn{\vv{\fstat}}
\def\uhatn{\vv{\uhat}}
\def\fhatn{\vv{\fhat}}
\def\umaxn{\mathbf{U}^{\max}}
\def\uminn{\mathbf{U}^{\min}}

\def\fmax{F^{\max}}
\def\fnmax{\vv{\fmax}}
\def\K{K}
\def\F{\mathcal{F}}
\def\thv{\bm{\theta}}

\def\pmin{\p^{\min}}
\def\pmax{\p^{\max}}
\def\J{J}
\newcommand{\Jk}[1]{\J_\mathrm{P#1}}
\newcommand{\Jkt}[2]{\J_{\mathrm{P#1},#2}}

\newcommand{\Pk}[1]{\text{\bf P#1}}

\def\ustat{\u^\mathrm{stat}}
\def\fstat{\f^\mathrm{stat}}

\def\gstat{\g^\mathrm{stat}}
\def\bs{\tilde \b}
\def\ks{\Gamma}
\def\ksmin{\ks^{\min}}
\def\ksmax{\ks^{\max}}
\def\W{W}
\def\Wmax{\W^{\max}}
\def\uhat{\u^\mathrm{ol}}
\def\fhat{\f^\mathrm{ol}}
\def\Dl{\underline{D}}
\def\Du{\overline{D}}
\def\M{M}
\def\Mone{\M^{\u}}
\def\Mtwo{\M^{\b}}
\def\Mthree{\M_{T}}
\def\Mfour{\mathcal M}
\def\N{\eta}
\def\None{\N^{\u}}
\def\Ntwo{\N^{\b}}
\def\L{L}
\def\Ls{\Delta}
\def\DT{\Delta T}
\def\R{N}
\def\renewrate{\mathcal{R}}
\def\gc{\phi}

\def\dmin{\d^{\min}}
\def\dmax{\d^{\max}}
\def\X{X}
\def\Xumin{\X^{\min,\u}}
\def\Xumax{\X^{\max,\u}}
\def\Xbmin{\X^{\min,\b}}
\def\Xbmax{\X^{\max,\b}}
\def\Xudot{\X^{(\cdot),\u}}

\def\Xbdot{\X^{(\cdot),\b}}

\newcommand{\rev}[1]{{\color{blue} #1}}
\def\Tday{\mathcal{T}^\mathrm{Day}}

\def\IM{A}
\def\be{\beta}
\def\ben{\bm{\beta}}
\def\thet{\theta}
\def\thetn{\bm{\theta}}

\usepackage{fullpage}
\usepackage{cite}

\renewcommand{\rev}[1]{{#1}}
\begin{document}\date{}

\title{Control of Generalized Energy Storage Networks
\thanks{This report, written in January 2014, is a longer version of the conference paper \cite{QCYR:acm}. An extended treatment of the single storage case be found at \cite{OMGarxiv}. An improved algorithm for the networked case and its distributed implementation can be found at \cite{TSGarxiv}. This version contains a somewhat more general treatment for the cases with sub-differentiable objective functions and Markov disturbance. }}
\author{Junjie Qin, Yinlam Chow, Jiyan Yang, Ram Rajagopal
\thanks{J. Qin, Y. Chow and J. Yang are with the Institute for Computational and Mathematical Engineering, Stanford, CA, 94305. Email: \texttt{\{jqin, ychow, jiyan\}@stanford.edu}.
R. Rajagopal is with Department of Civil and Environmental Engineering, Stanford, CA, 94305. Email: \texttt{ramr@stanford.edu}.}
}

\maketitle

\begin{abstract}
The integration of intermittent and volatile renewable energy resources requires increased flexibility in the operation of the electric grid. Storage, broadly speaking, provides the flexibility of shifting energy over time; network, on the other hand,  provides the flexibility of shifting energy over geographical locations. The optimal control of general storage networks in uncertain environments is an important open problem.
The key challenge is that, even in small networks, the corresponding constrained stochastic control problems with continuous spaces suffer from curses of dimensionality, and are intractable in general settings.
For large networks, no efficient algorithm is known to give optimal or near-optimal performance.

This paper provides an efficient and provably near-optimal algorithm to solve this problem in a very general setting.
We study the optimal control of generalized storage networks, \ie,  electric networks connected to distributed generalized storages. Here generalized storage is a unifying dynamic model for many components of the grid that provide the functionality of shifting energy over time,  ranging from standard energy storage devices to deferrable or thermostatically controlled loads. An online algorithm is devised for the corresponding constrained stochastic control problem based on the theory of Lyapunov optimization. We prove that the algorithm is near-optimal, and construct a semidefinite program to minimize the sub-optimality bound. The resulting bound is a constant that depends only on the parameters of the storage network and cost functions, and is independent of uncertainty realizations.  Numerical examples are given to demonstrate the effectiveness of the algorithm.
\end{abstract}

%
\newtheorem{assumption}[theorem]{Assumption}
\newtheorem{definition}[theorem]{Definition}
\newtheorem{remark}[theorem]{Remark}
%

\addtolength{\abovedisplayskip}{.02in}
\addtolength{\belowdisplayskip}{.02in}

\section{Introduction}
To ensure a sustainable energy future, deep penetration of renewable energy generation is essential. Renewable energy resources, such as wind and solar, are intrinsically variable. Uncertainties associated with these intermittent and volatile resources pose a significant challenge to their integration into the existing grid infrastructure \cite{NRELWest2010}. More flexibility, especially in shifting energy supply and/or demand across time and network, are desired to cope with the increased uncertainties.

Energy storage provides the functionality of shifting energy across time. A vast array of technologies, such as  batteries, flywheels, pumped-hydro, and compressed air energy storages, are available for such a purpose \cite{Denholm2010, lindley2010naturenews}.
Furthermore, flexible or controllable demand provides another ubiquitous source of storage. Deferrable loads -- including many thermal loads, loads of internet data-centers and loads corresponding to charging electric vehicles (EVs) over certain time interval \cite{thermalStor1993} -- can be interpreted as \emph{storage of demand} \cite{ObRACC2013}. 
Other controllable loads which can possibly be shifted to an earlier or later time, such as thermostatically controlled loads (TCLs), may be modeled and controlled as a storage with negative lower bound and positive upper bound on the storage level \cite{HaoSanandajiPoollaVincent2013}.
These forms of storage enable inter-temporal shifting of excess energy supply and/or demand, and significantly reduce the reserve requirement and thus system costs. 

On the other hand, shifting energy across a network, \ie, moving excess energy supply to meet unfulfilled demand  among different geographical locations with transmission or distribution lines, can achieve similar effects in reducing the reserve requirement for the system.
Thus in practice, it is natural to consider these two effects together. Yet, it remains mathematically challenging to formulate a sound and tractable problem that accounts for these effects in electric grid operations. Specifically, due to the power flow and network constraints, control variables in connected buses are coupled. Due to the storage constraints, control variables in different time periods are coupled as well.
On top of that, uncertainties associated with stochastic generation and demand dramatically complicate the problem, because of the large number of recourse stages and the need to account for all probable realizations.

Two categories of approaches have been proposed in the literature. The first category is based on exploiting structures of specific problem instances, usually using dynamic programming. These structural results are valuable in providing insights about the system, and often lead to analytical solution of these problem instances.
However, such approaches rely heavily on specific assumptions of the type of storage, the form of the cost function, and the distribution of uncertain parameters. Generalizing results to other specifications and more complex settings is usually difficult, and consequently this approach is mostly used to analyze single storage systems.
For instance, analytical solutions to optimal storage arbitrage with stochastic price have been derived in  \cite{QRsimpleStorPes2012} without storage ramping constraints, and in \cite{MITrampStor} with ramping constraints. Problems of using energy storage to minimize energy imbalance are studied in various contexts; see \cite{SuEGTPS, RLDSACC, 2012arXiv1212.0272Q, 6672872} for reducing reserve energy requirements in power system dispatch, \cite{BitarRACC_colocated, Powell} for operating storage co-located with a wind farm, \cite{IBMload, DataCenter} for operating storage co-located with end-user demands, and \cite{StorDRLongbo} for storage with demand response.

The other category is to use heuristic algorithms, such as Model Predictive Control (MPC)\cite{XieEtAlWindStorMPC} and look-ahead policies \cite{NRELStorValue2013},  to identify sub-optimal storage control rules. Usually based on deterministic (convex) optimization, these approaches can be easily applied to general networks. The major drawback is that these approaches usually do not have any performance guarantee. Consequently,  it lacks theoretical justification for implementing them in real systems.
Examples of this category can be found in \cite{XieEtAlWindStorMPC} and references therein.



This work aims to bring together the best of both worlds, \ie, to design online deterministic optimizations that solve the stochastic control problem with provable guarantees. It contributes to the existing literature in the following ways. First, we formalize the notion of {\it generalized storage} as a dynamic model that captures a variety of power system components which provide the functionality of storage. Second, we formulate the problem of  storage network operation as a stochastic control problem with general cost functions, and provide examples of applications that can be encapsulated by such a formulation. Third, we devise an online algorithm for the problem based on the theory of Lyapunov optimization \rev{\footnote{\rev{Although closely related to the classical Lyapunov theory} \rev{for stability, the theory and techniques of Lyapunov optimization are relatively recent. See \cite{NeelyBook} for more details.}}}, and prove guarantees for its performance in terms of a bound of its sub-optimality. We also show that the bound is independent of the realizations of the uncertain parameters.  The bound is useful not only in assessing the worst-case performance of our algorithm, but also in evaluating the performance of other sub-optimal algorithms when the optimal costs are hard to obtain. It can also be used to estimate the maximum cost reduction that can be achieved by {\it any} storage operation, thus provides understanding for the limit of a certain storage system.
To the best of our knowledge, this is the first algorithm with provable guarantees for the storage operation problem with general electric networks.

Our methodology is closely related to that of \cite{DataCenter}, where the focus is on solving the problem of operating an idealized energy storage (with no energy dissipation over time, and no charging/discharging conversion loss) at data-centers. Our objective is to provide an algorithm to operate generalized storage network in a wide range of different settings. This requires an extended or a new analysis in the following aspects. From the modeling perspective, in order to capture applications such as deferrable loads and TCLs, we do not assume storage level is non-negative, instead, we only assume each storage is feasible (see Assumption~\ref{assume:feas} for more details).  Furthermore, modeling the dissipation of energy over time leads to a new sub-optimality bound; the bound in \cite{DataCenter} becomes a special case of our bound when the dissipation factor (or storage efficiency) is one.  A semidefinite program  is constructed to decide parameters of the algorithm in order to minimize the sub-optimality bound. Finally, the aspect of power network appears to be completely new.  

The rest of the paper is organized as follows. Section 2 formulates the problem of operating a generalized storage network under uncertainty. Section 3 gives the online algorithm and states the performance guarantee. Section 4 analyzes the single bus case in detail with a generalized storage, and Section 5 provides a summary of results for general storage networks. Numerical examples are then given in Section 6. Section 7 concludes the paper.

\section{Problem Formulation}\label{sec:problem}

\subsection{Generalized Storage Models}\label{sec:prob:singleBusModel}
We start by defining a generalized storage model for each fixed bus of the electric network. A diagram is shown in Figure~\ref{fig:diagSingle}. Such a model may be used for a single bus system by setting the network inflow to be zero, or as a component of an electric network as discussed in Section~\ref{sec:modelNetwork}. We work with a slotted time model, where $t$ is used as the index for an arbitrary time period. Given that the actual length of each time interval is constant, this allows for simple conversion from power units (\eg, MW) to energy units (\eg, MWh) and vice versa. Thus we assume all quantities under consideration in this paper are in energy units, albeit many power system quantities are conventionally specified in power units.

\begin{figure}[htbp]
\centerline{
\includegraphics[width = 0.60\textwidth]{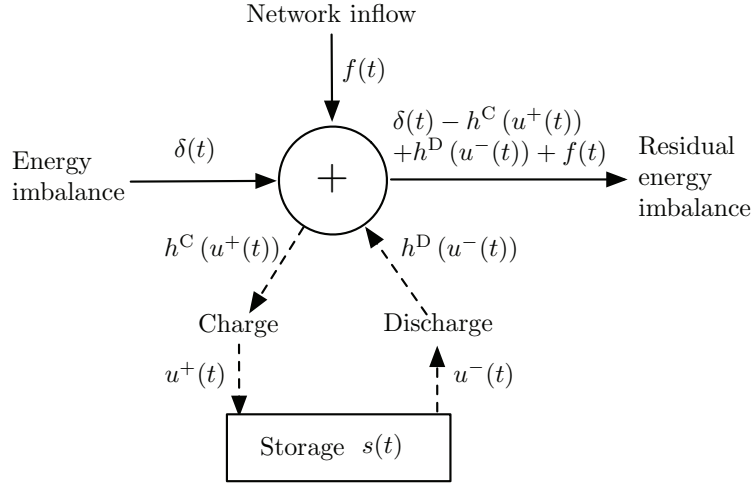}}
\caption{Diagram of a single-bus storage system.}
\label{fig:diagSingle}
\end{figure}

For the bus under consideration and time period $t$, the local \emph{energy imbalance} $\d(t)$ is defined to be the difference between the local generation and demand. Both the local generation and demand can be stochastic, and therefore $\d(t)$ is stochastic in general. The bus may be connected to other parts of the network, whose net energy inflow is denoted by $f(t)$. The bus is also connected to a \emph{generalized storage}, which is specified by the following elements:
\begin{itemize}
\item The \emph{storage level} or State of Charge (SoC) $\b(t)$ summarizes the status  of the storage at time period $t$. If $\b(t) \ge 0$, it represents the amount of energy in storage; if $\b(t)\le 0$, $-b(t)$ can represent the amount of currently deferred (and not fulfilled) demand.  It satisfies $\b (t) \in [\bmin, \bmax]$, where $\bmax$ is the storage capacity, and $\bmin$ is the minimum allowed storage level. 
\item The \emph{storage operation} $\u(t)$ summarizes the charging (when $\u(t)\ge 0$) and discharging (when $u(t)\le 0$) operations of the storage. It satisfies charging and discharging ramping constraints, \ie, $\u(t) \in [\umin, \umax]$, where $\umin \le 0$ whose magnitude is the maximum discharge within each time period, and $\umax \ge 0$ is the maximum charge within each time period. We also use $\upos(t) = \max(u(t),0)$ and $\uneg(t) = \max(-u(t),0)$ to denote the charging and discharging operations respectively.
\item The \emph{storage conversion function} $\h$ maps the storage operation $\u(t)$ into its effect on the bus. In particular, it is composed of two linear functions, namely the \emph{charging conversion function} $\hC$, and the \emph{discharging conversion function} $\hD$, such that the quantity $\hC(\upos(t))$ is the amount of energy that is withdrawn from the bus due to $\upos(t)$ amount of charge, and $\hD(\uneg(t))$ is the amount of energy that is injected into the bus due to $\uneg(t)$ amount of discharge, whence
\[\h(\u(t)) = \hD(\uneg(t))-\hC(\upos(t))\] is the net energy injection into the bus.
\item The \emph{storage dynamics} is then
\begin{equation}\label{eq:storDyn}
\b(t+1) = \la \b(t) + \u (t),
\end{equation}
where $\la\in (0,1]$ is the \emph{storage efficiency} which models the loss over time even if there is no storage operation.
\end{itemize}
The storage parameters satisfy the following consistency conditions.
\begin{assumption}[Feasibility]\label{assume:feas}
Starting from any feasible storage level, there exists a feasible storage operation such that the storage level in the next time period is feasible, that is
\begin{enumerate}
\item $\la \bmin + \umax \ge \bmin$,
\item $\la \bmax + \umin \le \bmax$.
\end{enumerate}
\end{assumption}
%
The \emph{residual energy imbalance}, after accounting for the network inflow and storage operation, is then given by:
\begin{equation}\label{eq:reEI}
\d(t) + \h(\u(t)) + f(t) = \d(t) - \hC(\upos(t)) + \hD(\uneg(t)) + f(t).
\end{equation}

We give a few examples of generalized storage models as follows.
\begin{example}[Storage of energy]\label{eg:soe}
Storage of energy can be modeled as a generalized storage with  $\bmax \ge \bmin \ge 0$. Here $\umin$ and $\umax$ correspond to the power rating of the storage, up to a multiple of the length of each time period.  By setting $\hC(\upos(t)) = (1/\muC)\upos(t)$, and $\hD(\uneg(t)) = \muD \uneg(t)$, one models the energy loss during charging and discharging operations. Here $\muC\in (0,1]$ is the charging efficiency; $\muD\in (0,1]$ is the discharging efficiency; and the round-trip efficiency of the energy storage is $\muC \muD$. For instance, based on the information from \cite{MS2006}, a NaS (sodium sulfur) battery can be modeled with parameters:
\[
\begin{split}
&(\bmin, \bmax, \umin, \umax, \muC, \muD, \la)=(0 \MWh, 100 \MWh, \\
&\quad\quad\quad\quad -10 \MW \times 1 \hr,10 \MW \times 1 \hr,  0.85, 0.85, 0.97),
\end{split}
\]
 and a CAES (compressed air energy storage) can be modeled with parameters:
 \[
 \begin{split}
& (\bmin, \bmax, \umin, \umax, \muC, \muD, \la)=(0 \MWh, 3000 \MWh,\\
 &\quad\quad\quad\quad-300 \MW \times 1 \hr, {300 \MW \times 1 \hr},0.85, 0.85, 1.00).
 \end{split}
 \]
\end{example}
\begin{example}[Storage of demand]\label{eg:sod}
Pre-emptive deferrable loads may be modeled as storage of demand, with $-\b(t)$ corresponding to the accumulated deferred (but not yet fulfilled) load up to time $t$ , and with $\u(t)$ corresponding to the amount of load to defer/fulfill in time period $t$. We have $\bmin \le \bmax \le 0$ in this case. Storage of demand differs from storage of energy in the sense that it has to be discharged before charging is allowed. 
The conversion function can usually be set to $\h(\u(t)) = \u(t)$,
and generally $\la = 1$ in deferrable load related applications.
\end{example}
\begin{example}[Generalized battery models]\label{eg:gbm}
It is shown recently that an aggregation of TCLs may be modeled as a generalized battery \cite{HaoSanandajiPoollaVincent2013}. A discrete time version of such a model can be cast into our framework by setting $\bmin = - \bmax$ and $\bmax \ge 0$. Other storage parameters can be set properly according to Definition 1 of \cite{HaoSanandajiPoollaVincent2013}, and we have $\la \le 1$ to model energy dissipation.
\end{example}

We consider the following \emph{stochastic piecewise linear cost function} for each fixed bus
\begin{align}\label{eq:singleBusGeneralCost}
\g(t) &= \sum_{\ell = 1}^L \p(t, \ell) \Big(\alI(\ell) \d(t) - \alC(\ell) \hC(\upos(t)) \\
&\quad + \alD(\ell) \hD(\uneg(t)) + \alF(\ell) f(t) + \alConst(t,\ell)\Big)^+, \nn
\end{align}
where the parameter $\p(t, \ell)$ is in general stochastic, and follows a  prescribed probability law,  and $\alI(\ell)$, $\alC(\ell)$, $\alD(\ell)$, $\alF(\ell)$ and $\alConst(t, \ell)$ are constants, for each $\ell = 1, \dots, L$ and $t$. This cost function serves as a generalization of positive (and/or negative) part cost function of the residual energy imbalance, and it encapsulates many applications of storage as shown in Section~\ref{sec:prob:applications}. Our analysis applies to a more general class of cost functions; see Appendix~\ref{appendix:convexCost} for more details.

\subsection{Applications in Single Bus Systems}\label{sec:prob:applications}
The storage operation problem on a single bus system ($f(t) = 0$) can be posed as an infinite horizon average cost stochastic control problem as follows:
\begin{subequations}\label{prob:singleBusGeneral}
\begin{align}
\minimize \quad &  \lim_{T \to \infty} \frac{1}{T} \expec \left[\sum_{t=1}^T \g(t) \right]\\
\st \quad & \b(t+1) = \la \b(t) + \u(t), \\
        & \bmin \le \b(t) \le \bmax, \\
        & \umin \le \u(t) \le \umax,
\end{align}
\end{subequations}
where we aim to find a control policy that maps the state $\b(t)$ to storage operation $\u(t)$, minimizes the expected average cost and satisfies all constraints for each time period $t$. Here, the initial state $\b(1) \in [\bmin, \bmax]$ is given.

Combining some specific cases of the generalized storage model given in Examples~\ref{eg:soe}-\ref{eg:gbm} with properly defined cost functions leads to possible problem instances of optimal control of storage under uncertainty. Here we provide examples that are considered in the literature.
\begin{example}[Balancing]
Storage may be used to minimize residual energy imbalance  given some stochastic $\{\d(t): t\ge 1\}$ process. Typical cost functions penalize the positive and negative residual energy imbalance differently, and may have different penalties at different time periods . (For example, to model the different consequences of load shedding at different times of the day.) The problem of optimal storage control for such a purpose can be modeled by problem~\eqref{prob:singleBusGeneral} with the cost function
\begin{align*}
\g(t) = &q^+(t) \pos{\d(t) - \hC(\upos(t)) +\hD(\uneg(t))} \\
&+ q^-(t) \neg{\d(t) - \hC(\upos(t)) +\hD(\uneg(t))}\\
= &q^+(t) \pos{\d(t) - \hC(\upos(t)) +\hD(\uneg(t))} \\
&+ q^-(t) \pos{-\d(t) + \hC(\upos(t)) -\hD(\uneg(t))},
\end{align*}
where $q^+(t)$ and $q^-(t)$ are the penalties\footnote{These penalties are usually prescribed deterministic sequences \cite{SuEGTPS}. } for each unit of positive and negative residual energy imbalance at time period $t$, respectively. 
\end{example}
\def\pl{\p^\mathrm{LMP}}
\begin{example}[Arbitrage]
Given that the locational marginal prices $\{\pl(t): t\ge 1\}$ are stochastic, a storage may be used to exploit arbitrage opportunities in electricity markets. The problem of maximizing the expected arbitrage profit using storage operations can be cast as an instance of \eqref{prob:singleBusGeneral}, with the cost function (\ie, negative profit) given by:
\begin{align*}
\g(t) = &\pl(t) (\hC(\upos(t)) - \hD(\uneg(t)))\\
= &\pl(t) \pos{\hC(\upos(t)) - \hD(\uneg(t))}\\
&-\pl(t) \pos{-\hC(\upos(t)) + \hD(\uneg(t))}.
\end{align*}
\end{example}
\begin{example}[Storage co-located with a stochastic generation or demand]
Applications of this type may be cast into our framework using $\{\d(t): t\ge 1\}$ to model the stochastic generation or demand process, and $\{\p(t,\ell): t\ge 1\}$  to model the stochastic prices. A possible cost function is 
\[
\g(t) = \pl(t) \pos{-\d(t) + \hC(\upos(t)) -\hD(\uneg(t))},
\]
where the residual energy is curtailed with no cost/benefit, and the residual demand is supplied via buying energy from the market at stochastic price $\pl(t)$. 
\end{example}

\subsection{Network Models}\label{sec:modelNetwork}
The electric network can be modeled as a directed graph $\G(\V, \E)$. Let $n = |\V|$, $m = |\E|$, and $\ER$ be the edge set with all edges reversed.
We use the notation $e\sim v$ to indicate that $e \in \{ (v',v) \in \E \cup \ER: v' \in \V\}$.
Each edge models a transmission (or distribution) line, and is associated with some power flow. Assuming the power system is operated in steady state, and the power flow is approximately a constant over each time period $t$, the energy flow through the line can be obtained by multiplying the power flow by the length of each time period and is denoted by $f_e(t)$ for $e\in \E$, with the direction of the edge indicating the positive direction of the flow.\footnote{To lighten the notation, for each $e = (v_1, v_2) \in \E$, we also define $\f_{e'}(t) = -\f_{e}(t)$ for $e' = (v_2, v_1)$, and therefore for all $v\in \V$, the net inflow $\sum_{e\sim v} f_e (t)  = \sum_{e= (v',v)\in \E} f_e (t)  + \sum_{ e = (v',v)  \in \ER} f_e (t) = \sum_{e= (v',v)\in \E} f_e (t)  - \sum_{ e = (v,v')  \in \E} f_e (t)$.} The flow vector $\fn(t) \in \real^m$ satisfies power flow constraints, which can be compactly summarized by the following set of linear constraints  using the classical DC power flow approximations to AC power flow equations \cite{Stott09}:
\begin{equation}\label{eq:Fcal}
\fn(t) \in \F,\quad \F = \{\fn\in \real^m: -\fnmax \le \fn \le \fnmax, \K\fn = 0 \},
\end{equation}
where $\fnmax\in \real^m$ is the vector of the line capacities of the network, and $\K\in \real^{(m-n+1)\times m}$ is a matrix summarizing the Kirchhoff's voltage law. The construction of this $\K$ matrix from network topology and line parameters can be found in Appendix~\ref{sec:DCflow}
Note that additional network constraints may be included in the definition of the set $\F$. 

Each node models a bus in the electric network. On bus $v\in \V$, a set of variables as described in Section~\ref{sec:prob:singleBusModel} is defined, with a subscript $v$ attached to each of the bus variables, and the network inflow is replaced by network flows to the bus from incident lines. The cost for bus $v$ and time period $t$ is then given by
\begin{align}\label{eq:networkGeneralCost}
\g_v(t) &= \sum_{\ell = 1}^{L_v} \p_v(t, \ell) \Big(\alI_v(\ell) \d_v(t) - \alC_v(\ell) \hC_v(\upos_v(t)) \\
&\quad\! +\! \alD_v(\ell) \hD_v(\uneg(t))\! +\! \alF_v(\ell) \sum_{e\sim v} f_e (t)\! +\! \alConst_v(t,\ell)\Big)^+, \nn
\end{align}
and the networked storage stochastic control problem is defined as follows:
\begin{subequations}\label{prob:networkGeneral}
\begin{align}
\minimize \quad &  \lim_{T \to \infty} \frac{1}{T} \expec \left[\sum_{t=1}^T\sum_{v\in \V} \g_v(t) \right]\label{prob:networkGeneral:obj}\\
\st \quad & \b_v(t+1) = \la_v \b_v(t) + \u_v(t), \label{prob:networkGeneral:storUpdate}\\
        & \bmin_v \le \b_v(t) \le \bmax_v, \\
        & \umin_v \le \u_v(t) \le \umax_v, \label{prob:networkGeneral:u}\\
        & \fn(t) \in \F.\label{prob:networkGeneral:F}
\end{align}
\end{subequations}
In this problem, we aim to find a control policy that maps the state $\bn(t)$ to storage operation $\un(t)$ and network flow $\fn(t)$, and minimizes the expected average cost objective function\eqref{prob:networkGeneral:obj}, such that constraints \eqref{prob:networkGeneral:storUpdate}-\eqref{prob:networkGeneral:u} hold for each $t$ and $v$, and \eqref{prob:networkGeneral:F} holds for each $t$.

\section{Online Algorithm and Performance Guarantees}
This paper provides an online algorithm for solving \eqref{prob:networkGeneral} with provable performance guarantees. Here we give a preview of the algorithm (Algorithm~\ref{alg}) and its sub-optimality bound (Theorem~\ref{thm:main}).  
\begin{algorithm}[htbp]
\caption{Online Lyapunov Optimization for Storage Network Control}
\begin{algorithmic}
\INPUT (i) Storage specifications $(\bmin_v$, $\bmax_v$, $\umin_v$, $\umax_v$, $\hC_v$, $\hD_v$, $\la_v)$, (ii) cost parameters in $\g_v(t)$ (including an upper bound and a lower bound on the sub-derivative of $\g_v(t)$ with respect to $\u_v(t)$, denoted by $\Du \g_v(t)$ and $\Dl \g_v(t)$, and  excluding any information about stochastic parameters $\d_v(t)$ and $\p_v(t, \ell)$, $\ell = 1, \dots, L_v$) for each bus $v\in \V$, and  (iii) network parameters $\K$ and $\fnmax$.
\Offline Determine algorithmic parameters $\ks_v$ and $\W_v$ for each bus $v\in \V$, by solving semidefinite program \eqref{prob:sdp_network}.
\Online 
\For {each time period $t$}
    \State Each bus $v\in \V$ observes realizations of stochastic parameters $\d_v(t)$ and $\p_v(t,\ell)$, $\ell = 1, \dots, L_v$.
    \State Solve the following deterministic optimization for storage operation $\un(t)$ and network flow $\fn(t)$:
\begin{subequations}\label{prob:P3_network}
\begin{align}
 \minimize  &  \sum_{v\in \V} \frac{\la_v (\b_v(t)+\ks_v) \u_v(t)}{\W_v} + \g_v(t)\\
\st  
        & \uminn \le \un(t) \le \umaxn, \\
        & \fn(t) \in \F.
\end{align}
\end{subequations}
\EndFor
\end{algorithmic}\label{alg}
\end{algorithm}
The performance theorem will hold under the following additional technical assumptions.
\begin{assumption}\label{assume:W_max}
For each bus $v\in \V$, the range of storage control is smaller than the effective capacity of the storage, \ie, $\umax_v - \umin_v < \bmax_v - \bmin_v$.
\end{assumption}
Since the bounds for storage control $\umax_v$ and $\umin_v$ are the product of the power rating of storage (in unit MW for example) and the length of each time period, this assumption holds for most systems as long as the length of each time period is made small enough. For instance, this assumption is satisfied for both energy storage examples in Example~\ref{eg:soe}.

We make the following assumption on the stochastic parameters of the system.
\begin{assumption}\label{assume:stoc}
Let the collection of stochastic parameters be $\thv(t) =\{\d_v(t),  \p_v(t, \ell), \ell = 1, \dots, L_v, v\in \V\} $. Then one of the following two assumptions is in force:
\begin{enumerate}
\item $\thv(t)$ is independent and identically distributed (i.i.d.) across time $t$, and is supported on a compact set.
\item $\thv(t)$ is some deterministic function of the system stochastic state $\omega(t)$, which is supported on an (arbitrarily large) finite set $\Omega$, and follows an ergodic Markov Chain.
\end{enumerate}
\end{assumption}
The first assumption is used to give a simple proof of the performance theorem and provide insights on the construction of our algorithm. 
The second alternative assumption intends to generalize the performance bounds to non-i.i.d. cases. 
The additional assumptions such as that $\omega(t)$ lies in a finite set are introduced to reduce the required technicality in view of the page limit. The same result can be obtained in a more general setting where $\{\thv(t): t\ge 1\}$ follows a renewal process. The performance bounds in these setting are the same, up to a multiple of the mean recurrence time of the stochastic process under consideration. 

Under Assumptions~\ref{assume:W_max} and~\ref{assume:stoc}, the following performance guarantee holds.
\begin{theorem}\label{thm:main}
The control actions $(\un(t), \fn(t))$ generated by Algorithm~\ref{alg} \rev{are} feasible for \eqref{prob:networkGeneral} and sub-optimal, whose sub-optimality\footnote{Here the sub-optimality is defined as the difference between the objective value of \eqref{prob:networkGeneral} with $(\un(t), \fn(t))$ generated by Algorithm~\ref{alg} and the optimal cost of \eqref{prob:networkGeneral}.} is bounded by a constant that depends only on the parameters of the storages and cost functions, and is independent of realizations of the stochastic parameters. 
\end{theorem}
The precise expressions of the sub-optimality bounds for the single bus case and general network case are given in Section~\ref{sec:proofSingleBus} and Appendix~\ref{sec:proofNetwork}, respectively, both under the i.i.d. assumption (Assumption~\ref{assume:stoc}.1). 
The bounds for  settings  with the Markov assumption (Assumption~\ref{assume:stoc}.2) are given in Appendix~\ref{appendix_markov_process}.
\begin{remark}[Convexity]
Our result holds without assuming that $\g_v(t)$ is convex in $\u_v(t)$, $v\in \V$. However, we do assume the online optimization \eqref{prob:P3_network} can be solved efficiently, and in all numerical examples we work with convex cost functions. 
\end{remark}


\section{Analysis for Single Bus Systems}\label{sec:proofSingleBus}
To demonstrate the proof ideas without unfolding all technicalities, we prove Theorem~\ref{thm:main} for a single bus system under the following simplifying assumptions.
\begin{assumption}\label{assume:singleBusProof}
We assume in this section:
\begin{itemize}
\item the imbalance process $\{\d(t): t\ge 1\}$ is independent and identically distributed (i.i.d.) across $t$ and is supported on a compact interval $[\dmin, \dmax]$;
\item for each $\ell = 1, \dots, L$, the process $\{\p(t,\ell): t \ge 1\}$ is i.i.d. across $t$ and is supported on a compact interval $[\pmin(\ell), \pmax(\ell)]$.
\end{itemize}
\end{assumption}
Define
\[
\bar \u \defeq \lim_{T \to \infty} \frac{1}{T} \expec\left[ \sum_{t=1}^T \u(t)\right],\,\,\bar \b \defeq \lim_{T \to \infty} \frac{1}{T} \expec\left[ \sum_{t=1}^T \b(t)\right].
\]
Note that for $\b(1)\in[\bmin,\bmax]$,
\[
\bar \u=\lim_{T \to \infty}\frac{1}{T}\expec\left[\sum_{t=1}^T\b(t+1)-\lambda\b(t)\right]=(1-\lambda)\bar\b.
\]
As $\b(t)\in[\bmin,\bmax]$ for all $t\geq 0$, the above expression implies
\[
(1-\lambda)\bmin\leq \bar \u\leq (1-\lambda)\bmax.
\]
Problem~\eqref{prob:singleBusGeneral} can be equivalently written as follows
\begin{subequations}\label{prob:P1}
\begin{align}
\opttag{P1:} \minimize  &  \lim_{T \to \infty} \frac{1}{T} \expec \left[\sum_{t=1}^T \g(t) \right]\\
\st  & \b(t+1) = \la \b(t) + \u(t), \label{P1:dynamics}\\
        & \bmin - \la \b(t) \le \u(t) \le \bmax - \la \b(t), \label{P1:bbounds-u}\\
        & \umin \le \u(t) \le \umax, \label{P1:ubounds}\\
        &(1-\lambda)\bmin\leq \bar \u\leq (1-\lambda)\bmax \label{P1:ubar},
\end{align}
\end{subequations}
where bounds on $\b(t)$ are replaced by \eqref{P1:bbounds-u}, and \eqref{P1:ubar} is added without loss of optimality.

The proof procedure is depicted in the diagram shown in Figure~\ref{fig:proofdiag}, where we use  $\J_\mathrm{P1}(\u)$ to denote the objective value of {\bf P1} with storage operation sequence $\u$ (as an abbreviation of $\{\u(t): t\ge 1\}$), $\u^\star(\mathrm{\bf P1})$ to denote an optimal sequence of storage operation for {\bf P1}, $\J^\star_\mathrm{P1} \defeq  \J_\mathrm{P1}(\u^\star(\mathrm{\bf P1}))$, and we define similar quantities for {\bf P2} and {\bf P3}. Here {\bf P2} is an auxilliary problem we construct to bridge the infinite horizon storage control problem {\bf P1} to online Lyapunov optimization problems {\bf P3} \eqref{prob:P3_network} (or \eqref{prob:P3} for single storage case). It has the following form
\begin{subequations}\label{prob:P2}
\begin{align}
\opttag{P2:} \minimize  &  \lim_{T \to \infty} \frac{1}{T} \expec \left[\sum_{t=1}^T \g(t) \right]\\
\st  
        & \umin \le \u(t) \le \umax, \label{P2:ubounds}\\
        &(1-\lambda)\bmin\leq \bar \u\leq (1-\lambda)\bmax \label{P2:ubar}.
\end{align}
\end{subequations}
Notice that it has the same objective as \Pk{1}, and evidently it is a relaxation of \Pk{1}. This implies that $\u^\star(\Pk{2})$ may not be feasible for \Pk{1}, and
\begin{equation}
\Jk{2}^\star = \Jk{1}(\u^\star(\Pk{2})) \le \Jk{1}^\star.
\end{equation}
The reason for the removal of state-dependent constraints \eqref{P1:bbounds-u} (and hence \eqref{P1:dynamics} as the sequence $\{\b(t): t\ge 1\}$ becomes irrelevant to the optimization of $\{\u(t): t\ge 1\}$) in \Pk{2} is that the state-independent problem \Pk{2} has easy-to-characterize optimal stationary control policies. In particular, from the theory of stochastic network optimization \cite{NeelyBook}, the following result holds.
\begin{lemma}
[Optimality of Stationary Disturbance-Only Policies]\label{lem:stat}
Under Assumption~\ref{assume:singleBusProof}
there exists a stationary disturbance-only\footnote{The policy is a pure function (possibly randomized) of the current disturbances $\d(t)$ and $\p(t,\ell)$, $\ell =1, \dots, L$.} policy $\ustat(t)$
, satisfying \eqref{P2:ubounds} and \eqref{P2:ubar},
and providing the following guarantees for all $t$:
\begin{align}
  & (1-\la)\bmin\leq \expec[\ustat(t) ] \leq (1-\la)\bmax,  \\
  & \expec[\g(t) | \u(t) = \ustat(t)]  = \Jk{2}^\star, \label{eq:statCostLink}
\end{align}
where the expectation is taken over the randomization of $\d(t)$, $\p(t,\ell)$, $\ell =1, \dots, L$, and (possibly) $\ustat(t)$.
\end{lemma}
Equation~\eqref{eq:statCostLink} not only assures the storage operation induced by the stationary disturbance-only policy achieves the optimal cost, but also
guarantees that the expected stage-wise cost is a constant across time $t$ and equal to the optimal time average cost. This fact will later be exploited in order to establish the performance guarantee of our online algorithm. By the merits of this Lemma, in the sequel, we overload $\u^\star(\Pk{2})$ to denote the storage operation sequence obtained from an optimal stationary disturbance-only policy.
\begin{figure}
\centerline{
\includegraphics[scale = .85]{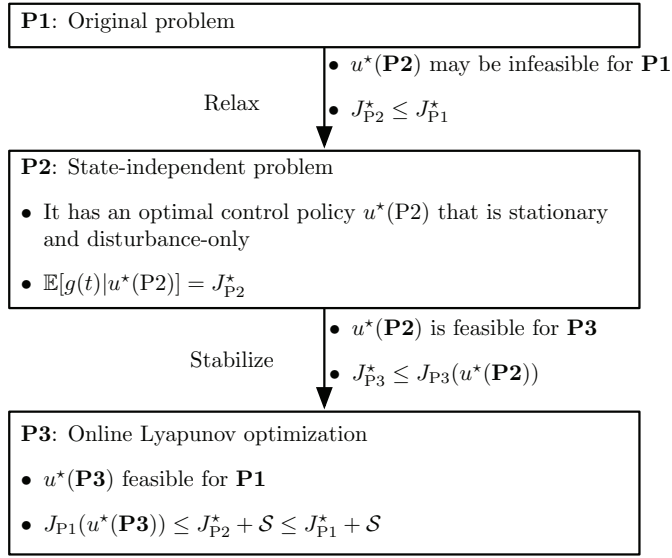}}
\caption{An illustration of the proof procedure as relations between problems considered in Section~\ref{sec:proofSingleBus}. Here $\mathcal{S}$ denotes the sub-optimality bound.}
\label{fig:proofdiag}
\end{figure}

An issue with $\u^\star(\Pk{2})$ for the original problem is that it may not be feasible for \Pk{1}. To have the $\{\b(t): t\ge 1\}$ sequence induced by the storage operation sequence lie in the interval $[\bmin, \bmax]$, we construct a virtual queue related to $\b(t)$ and use techniques from Lyapunov optimization to ``stabilize'' such a queue. Let the queueing state be a shifted version of the storage level:
\begin{equation}\label{eq:shift_queue}
\bs(t)=\b(t)+\ks,
\end{equation}
where the shift constant $\ks$ will be specified later. We wish to minimize the stage-wise cost $\g(t)$ and at the same time to maintain the queueing state close to zero.
This motivates us to consider solving the following optimization online (\ie, at the beginning of each time period $t$ after the realizations of stochastic parameters $\p(t,\ell)$, $\ell = 1, \dots, L$, and $\d(t)$ have been observed)
\begin{subequations}\label{prob:P3}
\begin{align}
\opttag{P3:} \minimize  &  \la \bs(t) \u(t) + \W \g(t)\label{eq:p3_obj}\\
\st  
        & \umin \le \u(t) \le \umax, \label{P3:ubounds}
\end{align}
\end{subequations}
where the optimization variable is $\u(t)$, the stochastic parameters in $\g(t)$ are replaced with their observed realizations, and $\W > 0$ is a weight parameter. Note that the objective here is a weighted combination of the stage-wise cost and a linear term of $\u(t)$, whose coefficient is positive when $\b(t)$ is large, and negative when $\b(t)$ is small.
We use the notation $\uhat(t)$ for the solution to \Pk{3} at time period $t$, $\u^\star(\Pk{3})$ for the sequence $\{\uhat(t): t\ge 1\}$, $\Jkt{3}{t}(\u(t))$ for the objective function of \Pk{3} at time period $t$, and $\Jkt{3}{t}^\star$ for the corresponding optimal cost.
In the rest of this section, we give conditions for parameters $\ks$ and $\W$ such that solving \Pk{3} online will result in a feasible $\{\b(t): t\ge 1 \}$ sequence (Section~\ref{sec:1busFeasible}), characterize the sub-optimality of $\u^\star(\Pk{3})$ as a function of $\ks$ and $\W$ and state the semidefinite program for identifying the optimal $\ks$ and $\W$ pair (Section~\ref{sec:1busPerformance}).

\subsection{Feasibility}\label{sec:1busFeasible}
We start with a structural result for the online optimization problem $\Pk{3}$. It follows from Lemma~\ref{lem:ana_sol} which is proved for general cost functions in Appendix~\ref{appendix:convexCost}.
\begin{lemma}\label{coro:tech_res}
At each time period $t$, the solution to \Pk{3}, $\uhat(t)$, satisfies
\begin{enumerate}
\item $\uhat(t) = \umin$ whenever $\la \bs(t) \ge - \W \Dl\g $,
\item $\uhat(t) = \umax$ whenever $\la \bs(t) \le - \W \Du\g $,
\end{enumerate}
where
\[
\Dl\g \defeq \inf\left\{\xi \in \partial_{\u(t)} \g(t)\middle|
\begin{array}{l}
 \u(t)\in [\umin,\umax],\,\, \\
\p(t, \ell)\in [\pmin(\ell), \pmax(\ell)],\, \forall \ell,\\
\d(t)\in [\dmin, \dmax],\\
\fn(t)\in\F, \,\, t \ge 1
\end{array}
\right\},
\]
and
\[
\Du\g \defeq \sup\left\{\xi \in \partial_{\u(t)} \g(t)\middle|
\begin{array}{l}
 \u(t)\in [\umin,\umax],\,\, \\
\p(t, \ell)\in [\pmin(\ell), \pmax(\ell)],\, \forall \ell,\\
\d(t)\in [\dmin, \dmax],\\
\fn(t)\in\F, \,\, t \ge 1
\end{array}
\right\}
\]
 are the greatest lower bound and the least upper bound of the sub-derivatives of $\g(t)$,
respectively. 
\end{lemma}
\begin{remark}[Evaluation of $\Du\g$, $\Dl\g$] Any finite lower bound and upper bound for the sub-derivative of the cost $\g(t)$ can be used as $\Dl \g$ and $\Du \g$, respectively. Here we use the greatest lower bound and least upper bound to provide the tightest performance bounds. For cases with simple cost functions, \eg, for idealized storage with $L = 1$, $\Dl \g$ and $\Du \g$ can be easily obtained from $\pmin (\ell)$,  $\pmax(\ell)$, and constants in the cost function $\g(t)$ (such as $\alC(\ell)$ and $\alD(\ell)$). In cases where $\g(t)$ is differentiable with respect to $\u(t)$, $\Du \g$ and $\Dl \g$ may be obtained by solving a simple optimization problem.
\end{remark}
This allows us to construct the following sufficient condition that will assure the feasibility of the $\{\b(t): t\ge 1\}$ sequence induced by $\u^\star(\Pk{3})$.
\begin{theorem}[Feasibility]\label{thm:feas_single}
Suppose the initial storage level satisfies $\b(1) \in [\bmin, \bmax]$, then the storage level sequence $\{\b(t): t\ge 1 \}$ induced by the sequence of storage operation $\u^\star(\Pk{3})$ is feasible with respect to storage level constraints, \ie, $\b(t) \in [\bmin, \bmax]$ for all $t$, provided that
\begin{align}
\ksmin\le & \ks \le \ksmax, \label{eq:ksbounds}\\
0 < & \W  \le \Wmax, \label{eq:Wbounds}
\end{align}
where
\begin{equation}\label{eq:ineq_1}
\!\!\ksmin\! = \frac{1}{\la} \left[\!-\W \Dl \g + \pos{\umax \!\!- (1-\la)\bmax}\right]\! -\! \bmax,\!\!\!\!\!\!
\end{equation}
\begin{equation}\label{eq:ineq_2}
\!\!\ksmax\! = \frac{1}{\la} \left[\!-\W \Du \g - \pos{(1-\la)\bmin \!\!- \umin}\right]\!- \!\bmin, \!\!\!\!\!\!
\end{equation}
and
\begin{align}\label{eq:W_max}
\!\!\Wmax\!\! =\! \frac{1}{\Du \g\! - \Dl \g} \Big[&\la (\bmax\! - \bmin)\! -\! \pos{(1-\la)\bmin\! -\! \umin}\nn\\
& \quad- \pos{\umax - (1-\la) \bmax}\Big].
\end{align}
\end{theorem}
\begin{proof}
The result is proved by induction, where Lemma~\ref{coro:tech_res} is used to partially characterize the $\uhat(t)$ sequence.
See Appendix~\ref{sec:app:1bus} for more details.
\end{proof}

\subsection{Performance}\label{sec:1busPerformance}
In the previous result, we have established that $\u^\star(\Pk{3})$ is feasible for \Pk{1} as long as parameters $\ks$ and $\W$ satisfy \eqref{eq:ksbounds} and \eqref{eq:Wbounds}. In the next theorem, we characterize the sub-optimality of $\u^\star(\Pk{3})$ for fixed $\ks$ and $\W$.
\begin{theorem} [Performance]\label{thm:perf_lyap}
The sub-optimality of storage operation $\u^\star(\Pk{3})$ is bounded by $\M(\ks)/\W$, that is
\begin{equation}\label{eq:1busiidperf_bdd}
\Jk{1}^\star \le  \Jk{1}(\u^\star(\Pk{3})) \le \Jk{1}^\star + \M(\ks)/\W,
\end{equation}
where
\begin{align*}
&\M(\ks)=\Mone(\ks)+\la(1-\la)\Mtwo(\ks),\\
&\Mone(\ks) =\! \frac{1}{2}\! \max\left(\! \left(\umin\!+(1-\la)\ks\right)^2\!\!,\left(\umax\!+(1-\la)\ks\right)^2\! \right)\!,\\
&\Mtwo(\ks)= \max\left( \left(\bmin+\ks\right)^2,\left(\bmax+\ks\right)^2 \right).
\end{align*}
\end{theorem}
\begin{proof}
A quadratic Lyapunov function is constructed. The relation between the Lyapunov drift and the objective of \Pk{3} is exploited, which in turn relates to the objective of \Pk{2} and so \Pk{1}. Appendix~\ref{sec:app:1bus} contains the whole proof.
\end{proof}


The theorem above guarantees that the worst-case cost (among different uncertainty realizations) of our online algorithm is bounded above by $\Jk{1}^\star + \M(\ks)/\W$. The sub-optimality bound $\M(\ks)/\W$ reduces to a much simpler form if $\la = 1$.
\begin{remark}[Sub-Optimality Bound, $\la =1 $]\label{remark:subo:la1}
For a storage with $\la = 1$, we have \[\M \defeq \M(\ks) = (1/2) \max((\umin)^2, (\umax)^2),\] and the online algorithm is no worse than $\M/\W$ sub-optimal.
 In this case, one would optimize the performance by setting
\[
\W=\Wmax = \frac{(\bmax - \bmin)-(\umax-\umin)}{\Du \g\! - \Dl \g},
\]
and the corresponding
interval $[\ksmin,\ksmax]$ turns out to be a singleton, where
\[
\ksmin=\ksmax=-\frac{\Du \g(\bmax-\umax)+\Dl \g(\umin-\bmin)}{\Du \g\! - \Dl \g}.
\]
Let $\bmax - \bmin = \rho (\umax - \umin)$. Suppose $|\umax| = |\umin|$. For efficient storage ($\la = 1$),  the sub-optimality bound is
\[
\frac{\M}{\W}= \frac{(1/2)(\Du \g - \Dl g)(\umax)^2}{(\bmax- \bmin) - (\umax - \umin)} = \frac{\Du \g - \Dl g}{4(\rho - 1)} \umax.
\]
For fixed $\umax$, as storage capacity increases , \ie,  $\rho \to \infty$, the sub-optimality $(\M/\W) \to 0$. If $\umax$ and $\bmax$ increases with their ratio $\rho$ fixed, the bound increases linearly with $\umax$.
\end{remark}

The remaining  case $\la \in (0,1)$ requires solving an optimization program to identify the bound-minimizing parameter pair $(\ks, \W)$. In the next result, we state a semidefinite program to find $(\ks^\star, \W^\star)$ that solves the following  parameter optimization program
\begin{subequations}\label{P3:PO}
\begin{align*}
\opttag{P3-PO:}\minimize & \quad \M(\ks)/\W\\
\st & \quad  \ksmin\le \ks \le \ksmax,\,\, 0< \W \le \Wmax.
\end{align*}
\end{subequations}
In the current form, this program appears to be non-convex. The next result reformulates {\bf P3-PO} into a semidefinite program. Note that $\ksmin$ and $\ksmax$ are linear functions of $\W$ as defined in \eqref{eq:ineq_1} and \eqref{eq:ineq_2}.

\begin{lemma}[Semidefinite Reformulation of \Pk{3}-{\bf PO}]\label{SDP_P3_PO}
Let symmetric positive definite matrices $\Xumin$, $\Xumax$, $\Xbmin$, and $\Xbmax$ be defined as follows
\begin{equation*}
\!\!\Xudot\!\! = \!\!\begin{bmatrix}
\None&\!\!\udot+(1-\la)\ks\\
*& 2\W
\end{bmatrix}, \,\,
\Xbdot\!\! = \!\!\begin{bmatrix}
\Ntwo&\!\!\bdot+\ks\\
*& \W
\end{bmatrix},\!\!
\end{equation*}
where $(\cdot)$ can be either $\max$ or $\min$, and $\None$ and $\Ntwo$ are auxilliary variables. Then \Pk{3}-{\bf PO} can be solved via the following semidefinite program
\begin{subequations}\label{prob:sdp}
\begin{align}
\!\!\emph{\minimize} \quad & \None+\la(1-\la)\Ntwo \!\!\\
\!\!\emph{\st} \quad & \ksmin\le  \ks \le \ksmax,\,\, 0 <  \W  \le \Wmax, \!\!\\
& \Xumin,\Xumax, \Xbmin, \Xbmax \succeq 0. \!\!
\end{align}
\end{subequations}
\end{lemma}
\begin{proof}
The result follows from Schur complement. See Appendix~\ref{sec:app:1bus} for details.
\end{proof}

We close this section by discussing several implications of the performance theorem.
\begin{remark}[Optimality at the Fast-Acting Limit]
Let the length of each time period be $\Delta t$.
At the limit $\Delta t \to 0$, the online algorithm is optimal. Indeed, as discussed in Section~\ref{sec:problem}, both $|\umin|$ and $|\umax|$ are linear in $\Delta t$, such that $|\umax| \to 0$ and $|\umin|\to 0 $ as $\Delta t\to 0$. Meanwhile, $\la \to 1$ as $\Delta t \to 0$. So by Remark~\ref{remark:subo:la1}, it is easy to verify that the sub-optimality $\M/\W$ converges to zero as $\Delta t \to 0$.
\end{remark}

\begin{remark}[Operational Value of Storage and Percentage Cost Savings]\label{remark:vos}
Operational Value of Storage (VoS) is broadly defined as the savings in the long term system cost due to storage operation. Such an index is usually calculated by assuming storage is operated optimally. In stochastic environments, the optimal system cost with storage operation is hard to obtain in general settings. In our notations, let $\u^\mathrm{NS}$ denote the sequence $\{\u(t): \u(t) = 0,  t\ge 1 \}$ which corresponds to no storage operation. Then
\[
\mathrm{VoS} = \Jk{1}(\u^\mathrm{NS}) - \Jk{1}^\star,
\]
and it can be estimated by the interval 
\[
\left[\Jk{1}(\u^\mathrm{NS})\! -\! \Jk{1}(\u^\star(\Pk{3})),\,\,  \Jk{1}(\u^\mathrm{NS})\! -\! \Jk{1}(\u^\star(\Pk{3})) \!+ \!\frac{\M}{\W} \right].
\]
Additionally, for a storage operation sequence $\u$, the percentage cost savings due to storage can then be defined by $(\Jk{1}(\u^\mathrm{NS}) - \Jk{1}(\u))/\Jk{1}(\u^\mathrm{NS})$. An upper bound of this for any storage control policy can be obtained via $(\Jk{1}(\u^\mathrm{NS}) - \Jk{1}(\u^\star(\Pk{3})) + \M/\W)/\Jk{1}(\u^\mathrm{NS})$, which to an extent summarizes the limit of a storage system in providing cost reduction.
\end{remark}

\section{Numerical Experiments}
\subsection{Single Storage Example}
We first test our algorithm in a simple setting where the analytical solution for the optimal control policy is available, so that the algorithm performance can be compared against the true optimal costs. We consider the problem of using a single energy storage to minimize the energy imbalance as studied in \cite{SuEGTPS}, where it is shown that greedy storage operation is optimal if $\la = 1$ and if the following cost is considered
\[
\g (t)  = |\d(t) - (1/\muC)\upos(t) + \muD\uneg(t)|.
\]
As in \cite{SuEGTPS}, we specify storage parameters in per unit, and  $\bmin = 0$. Let  $\muC = \muD = 1$ so that the parameterization of storage operation here is equivalent to that of \cite{SuEGTPS}. We assume each time period represents an hour, and $-\umin = \umax = (1/10) \bmax$.
In order to evaluate the performance, we simulate the $\d(t)$ process by drawing i.i.d. samples from zero-mean Laplace distribution, with standard deviation $\sigma_\d = 0.149$ per unit obtained from NREL data \cite{SuEGTPS}. \rev{The time horizon for the simulation is chosen to be $T=1000$.} Figure~\ref{fig:result_1stor_a} depicts the performance of the our algorithm and the optimal cost $\Jk{1}^\star$ obtained from the greedy policy,  where it is shown that the algorithm performance is near-optimal, and better than what the (worst-case) sub-optimality bound predicts. \footnote{By an abuse of notation, in this section, we use $\Jk{1}^\star$ and $\Jk{1}$ to denote the results from simulation, which are estimates of the true expectations.}

A slight modification of the cost function would render a problem which does not have an analytical solution. Consider the setting where only unsatisfied demand is penalized with a higher penalty during the day ($7$ am to $7$ pm):
\begin{equation}\label{eq:inhomo:cost}
\!\!\g (t)\!\! = \!\!
\begin{cases}
3\neg{\d(t)\! -\! (\upos(t)/\muC) + \muD\uneg(t)}\!\!,\!\! &t\in \Tday,  \\
\neg{\d(t)\! -\! (\upos(t)/\muC) + \muD\uneg(t)}\!\!, \!\!& \mbox{otherwise},
\end{cases}
\end{equation}
where $\Tday = \{t\ge 1: 7 \le  t\bmod 24 < 19\}$.
We run the same set of tests above, with the modification that now $\muC = \muD = 0.95$. Note that the greedy policy is only a sub-optimal heuristic for this case. Figure~\ref{fig:result_1stor_b} shows our online algorithm performs significantly better than the greedy algorithm. The costs of our algorithm together with the lower bounds give a narrow envelope for the optimal average cost $\Jk{1}^\star$ in this setting, which can be used to evaluate the performance of other sub-optimal algorithms numerically. In both experiments, we also plot the costs of predictive/nominal storage control, whose solution can be shown to be $\u(t) = 0 $ for all $t$. Consequently, the costs of such operation rule are the same as the system costs when there is no storage.
\begin{figure}[htbp]
\centering
\subfigure[]{ \label{fig:result_1stor_a}
\centering
\psfrag{SM}[cc][Bl]{$\bmax$}
\psfrag{AverageCost}[cc][Bt]{Average cost $\Jk{1}$}
\psfrag{NoStorage}{\scriptsize No storage}
\psfrag{Lyapunov}{\scriptsize Lyapunov}
\psfrag{Greedy}{\scriptsize Greedy ($\Jk{1}^\star$)}
\psfrag{LowerBoundLegend}{\scriptsize Lower bound}
\includegraphics[width=.33\textwidth]{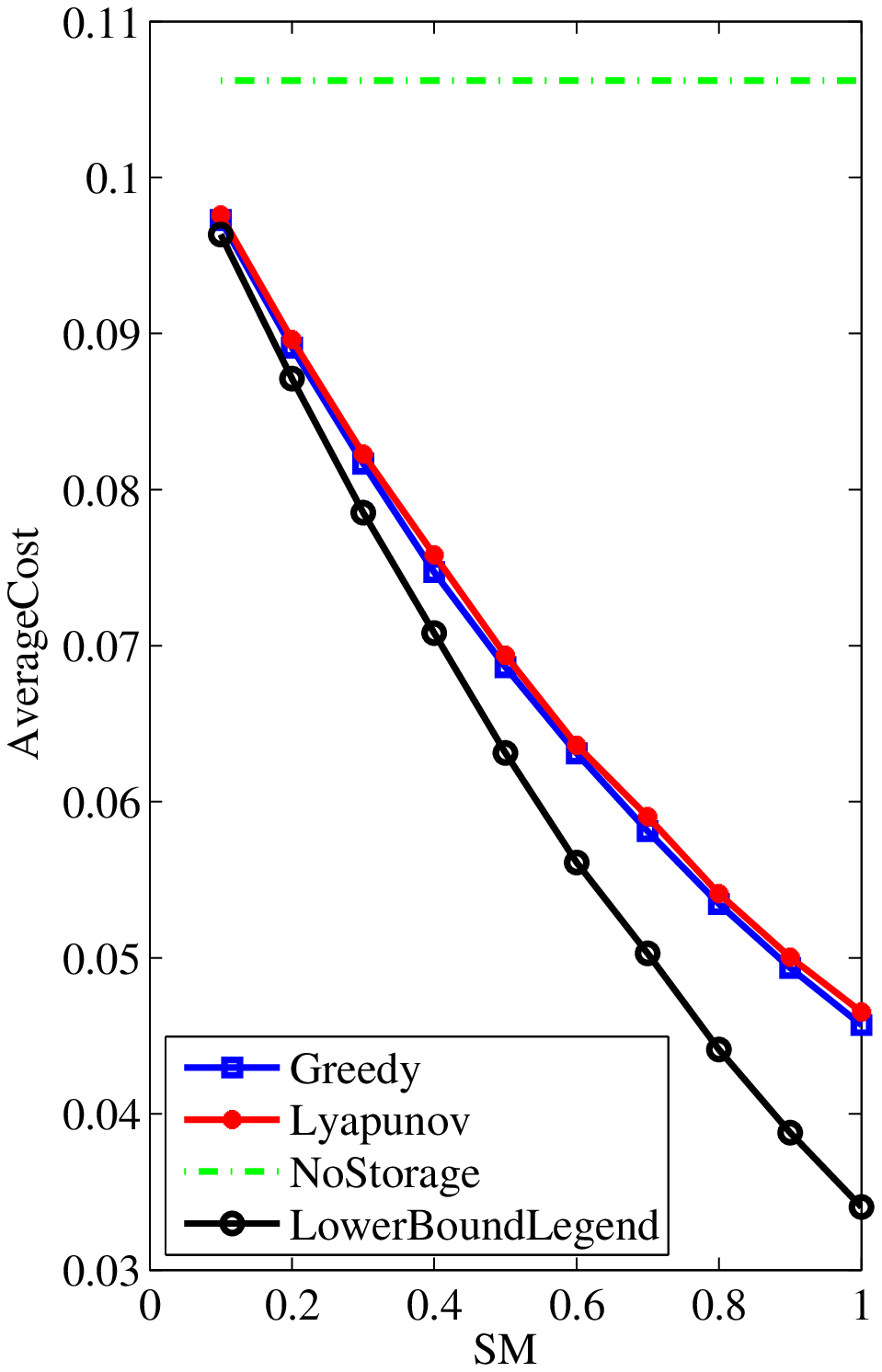}}
\subfigure[]{ \label{fig:result_1stor_b}
\centering
\psfrag{SM}[cc][Bl]{$\bmax$}
\psfrag{AverageCost}[cc][Bt]{Average cost $\Jk{1}$}
\psfrag{NoStorage}{\scriptsize No storage}
\psfrag{Lyapunov}{\scriptsize Lyapunov}
\psfrag{Greedy}{\scriptsize Greedy}
\psfrag{LowerBoundLegend}{\scriptsize Lower bound}
\includegraphics[width=.33\textwidth]{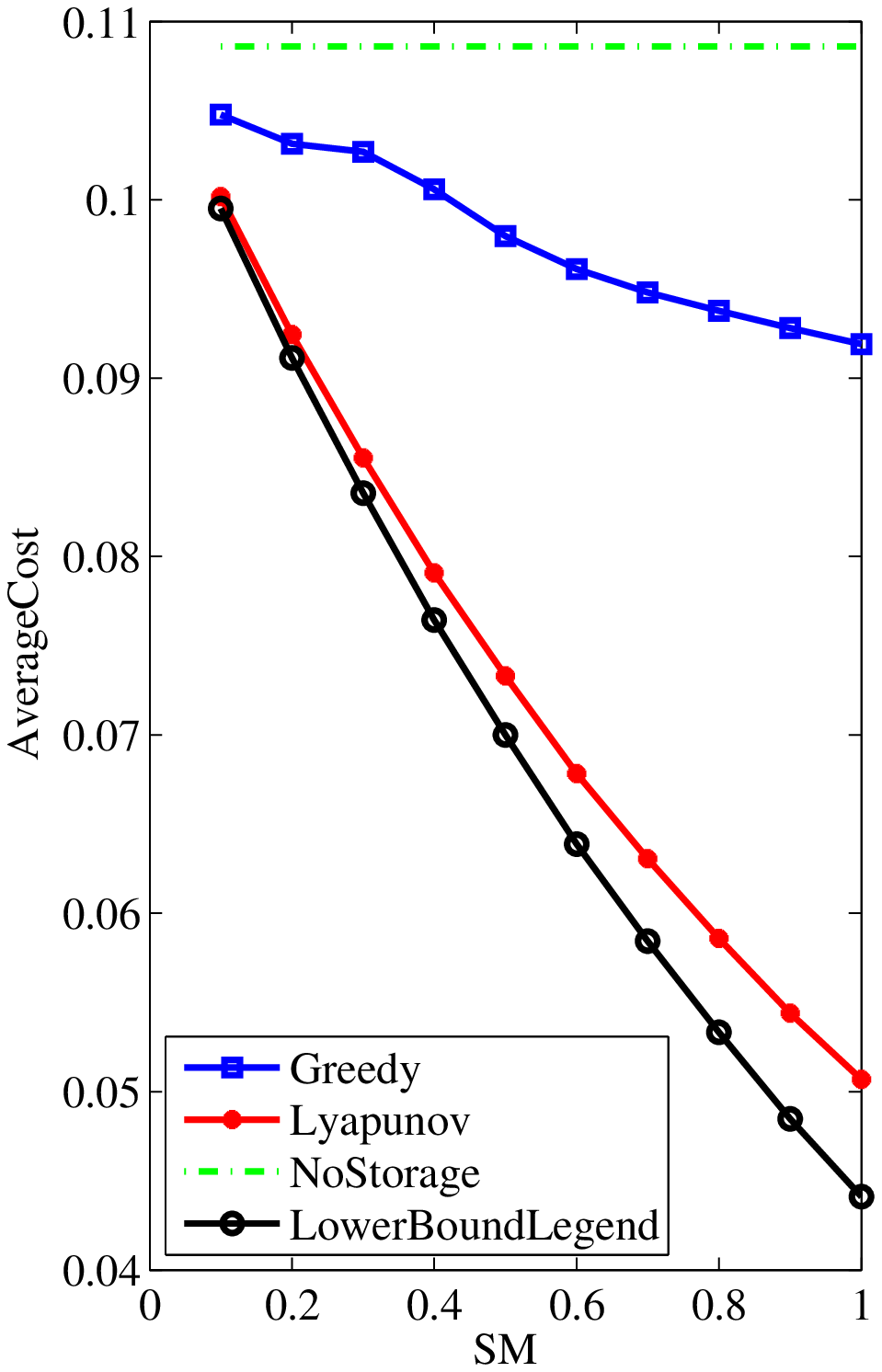}}
\caption{Performance of algorithms in a single bus network. \rev{Note that, different from the setup in Remark~\ref{remark:subo:la1}, we scale $\umin$, $\umax$ together with $\bmax$ in this and the following numerical examples.}}\label{fig:3}
\end{figure}

Figure~\ref{fig:4} translates the results in Figure~\ref{fig:3} into percentage cost savings due to storage operation, which are computed following the discussion in Remark~\ref{remark:vos}. Using the cost of our online algorithm and the theoretical sub-optimality bound, we obtain an upper bound of percentage cost reduction of energy storage for {\it any control policy} (see black curve in each panel of Figure~\ref{fig:4}). It indicates the systemic limit of using storage to provide cost reduction,  and is useful for system design considerations especially when the optimal cost cannot be calculated efficiently.
\begin{figure}[htbp]
\centering
\subfigure[]{ \label{fig:result_1stor_perc_a}
\centering
\psfrag{SM}[cc][Bl]{$\bmax$}
\psfrag{PercentageSaving}[cc][Bt]{Percentage cost savings (\%)}
\psfrag{NoStorage}{\scriptsize No storage}
\psfrag{Lyapunov}{\scriptsize Lyapunov}
\psfrag{Greedy}{\scriptsize Greedy ($\Jk{1}^\star$)}
\psfrag{LowerBoundLegend}{\scriptsize Lower bound}
\psfrag{UpperBoundLegend}{\scriptsize Upper bound}
\includegraphics[width=.33\textwidth]{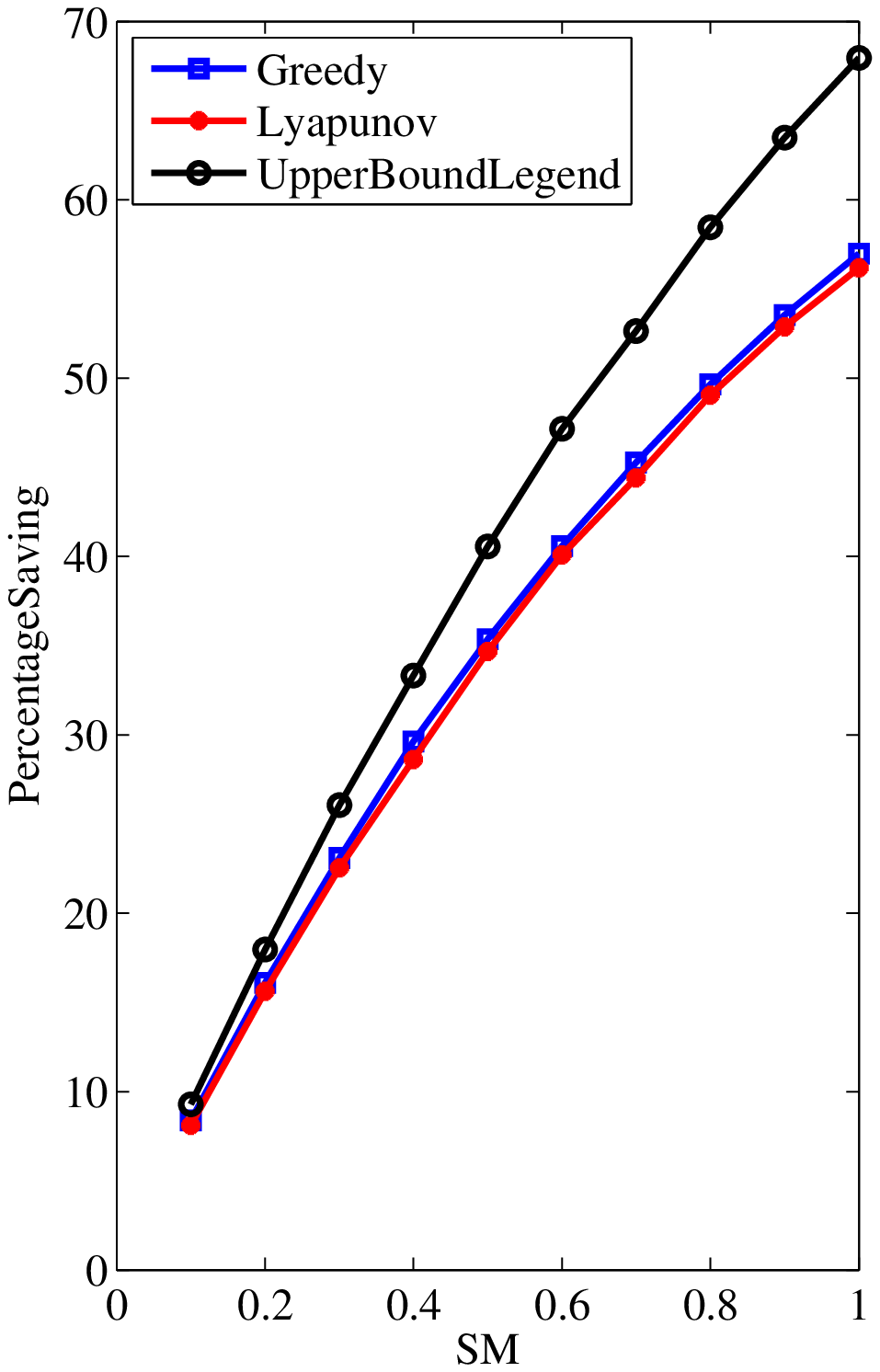}}
\subfigure[]{ \label{fig:result_1stor_perc_b}
\centering
\psfrag{SM}[cc][Bl]{$\bmax$}
\psfrag{PercentageSaving}[cc][Bt]{Percentage cost savings (\%)}
\psfrag{NoStorage}{\scriptsize No storage}
\psfrag{Lyapunov}{\scriptsize Lyapunov}
\psfrag{Greedy}{\scriptsize Greedy}
\psfrag{LowerBoundLegend}{\scriptsize Lower bound}
\psfrag{UpperBoundLegend}{\scriptsize Upper bound}
\includegraphics[width=.33\textwidth]{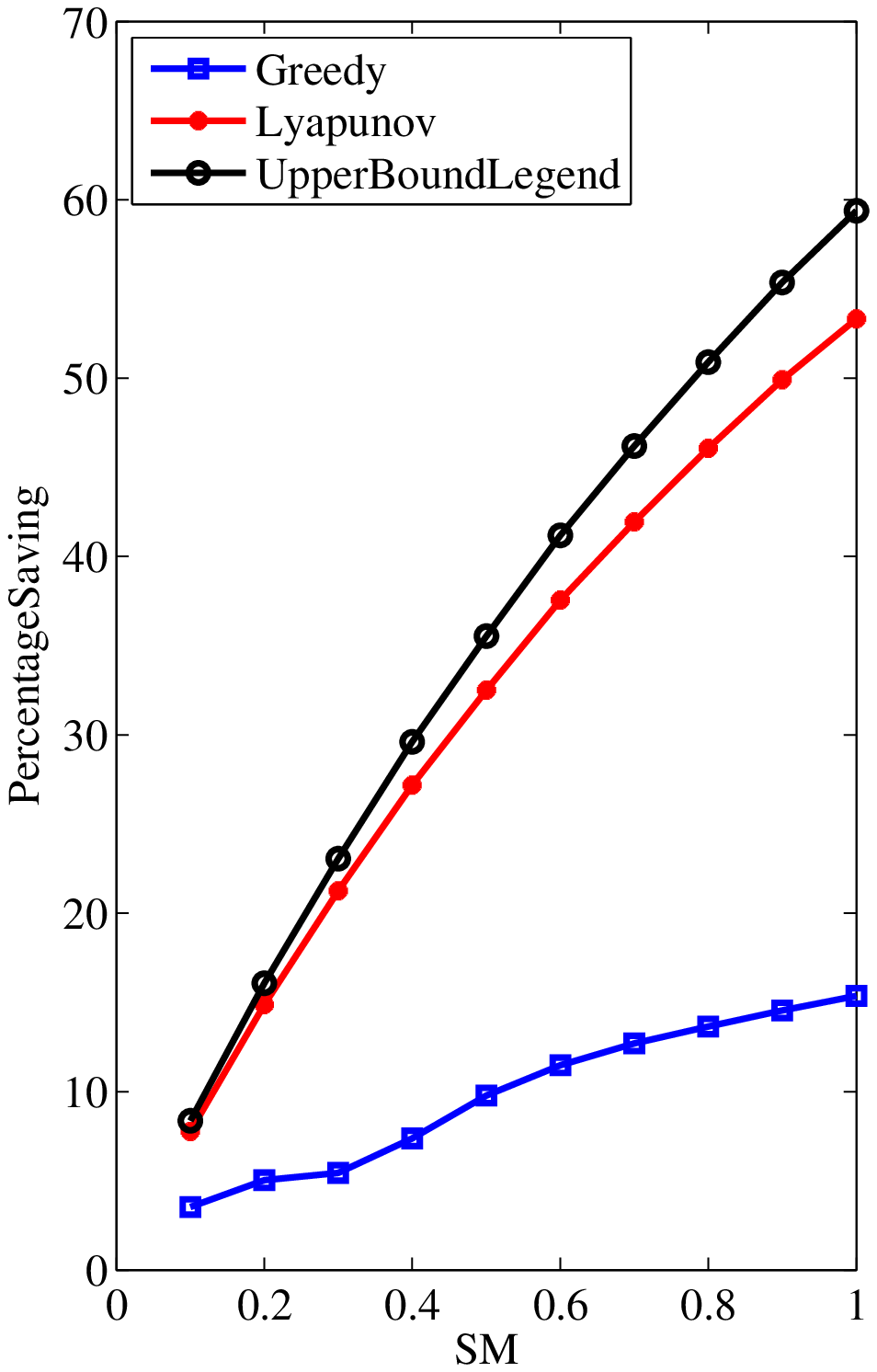}}
\caption{Percentage cost savings of a single storage operated for balancing.}\label{fig:4}
\end{figure}

\subsection{Storage Network Example}
We consider a setting similar to that in the single storage numerical example, in which now distributed storages are coordinated to minimize the power imbalance over a tree network with $N$ buses. We assume the storage network is homogeneous, \ie, the storage installed on each bus of the network has the same specifications and the same cost functions. Two cases with different cost functions are considered. In the first case, time homogeneous costs of the form
\begin{equation} \label{eq:cost:net:homo}
\g^\mathrm{H}_v(t) = \neg{\d_v(t) - (1/\muC_{v})\upos_{v}(t) + \muD_v\uneg_{v}(t) + \sum_{e\sim v} f_e (t)},
\end{equation}
are considered,
where $\muC_v = \muD_v = 0.95$, and
$\d_v(t)$ is i.i.d. following the same distribution as in the single storage example. In the second case, each bus has a cost function similar to \eqref{eq:inhomo:cost}:
\begin{equation*}
\g_v (t) = 
\begin{cases}
3 \g^\mathrm{H}_v(t), &t\in \Tday,  \\
\g^\mathrm{H}_v(t) , & \mbox{otherwise},
\end{cases}
\end{equation*}
with $\g^\mathrm{H}_v(t)$ as defined in \eqref{eq:cost:net:homo}.
We consider non-idealized storages which are operated frequently such that $\la_v = 0.999$ for all  $v\in \V$. As in the single storage example, we fix $-\umin_v = \umax_v = (1/10) \bmax_v$. The storages are connected in a star network, with $N= 5$ and $\fmax_e = \sigma_\d$ for each line $e\in \E$. \rev{The time horizon for the simulation is chosen to be $T=1000$.} Figure~\ref{fig:5} shows the percentage cost savings, where it is demonstrated that the online algorithm performs consistently superior to the greedy heuristic, and leads to percentage cost savings values that are close to the derived upper bound. Therefore near-optimal performance is achieved by our algorithm in both cases.

\begin{figure}[htbp]
\centering
\subfigure[]{ \label{fig:result_Nstor_perc_a}
\centering
\psfrag{SMR}[cc][Bl]{Total $\bmax$ in network}
\psfrag{PercentageSaving}[cc][Bt]{Percentage cost savings (\%)}
\psfrag{NoStorage}{\scriptsize No storage}
\psfrag{Lyapunov}{\scriptsize Lyapunov}
\psfrag{Greedy}{\scriptsize Greedy}
\psfrag{LowerBoundLegend}{\scriptsize Lower bound}
\psfrag{UpperBoundLegend}{\scriptsize Upper bound}
\includegraphics[width=.33\textwidth]{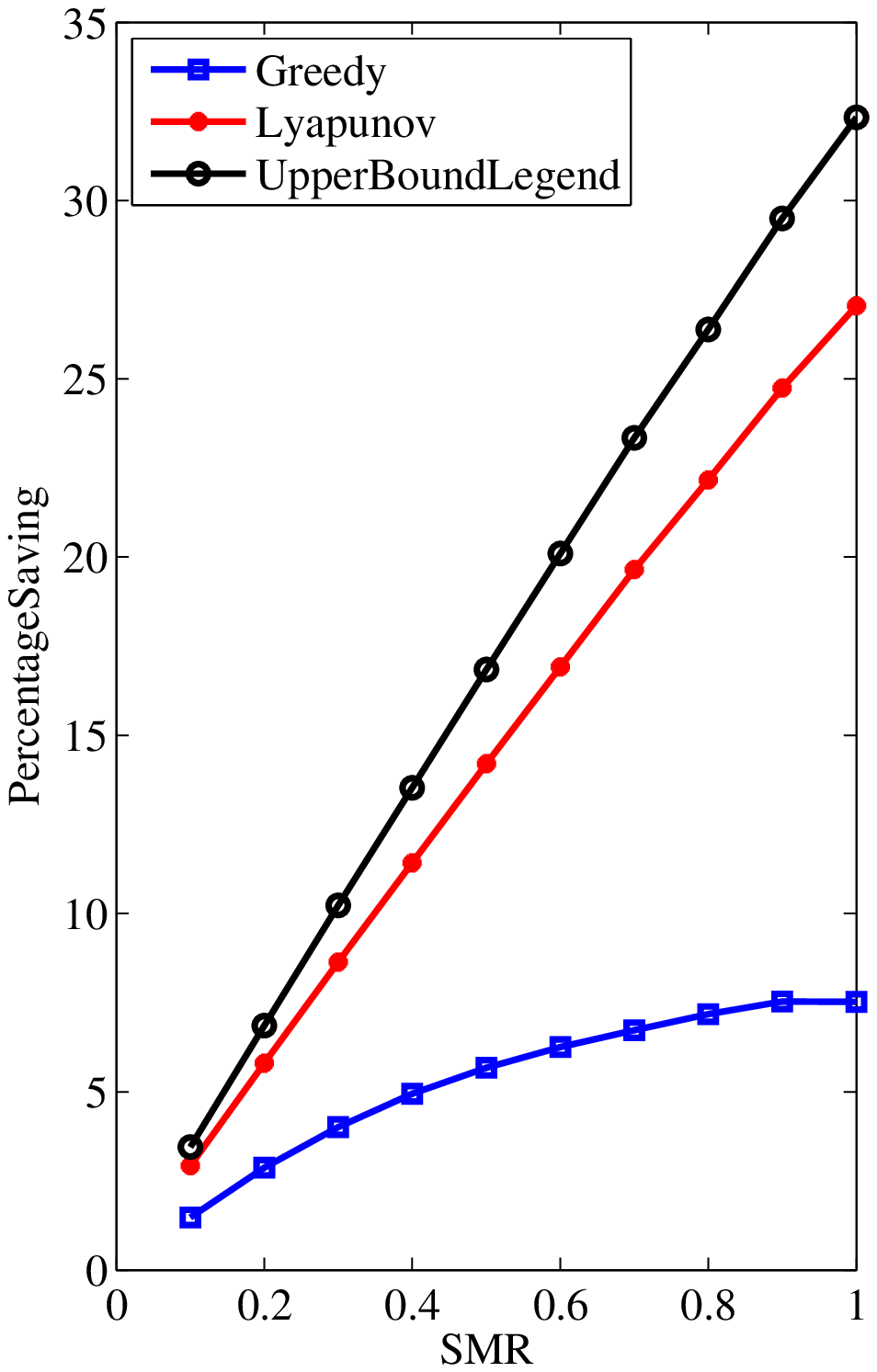}}
\subfigure[]{ \label{fig:result_Nstor_perc_b}
\centering
\psfrag{SMR}[cc][Bl]{Total $\bmax$ in network}
\psfrag{PercentageSaving}[cc][Bt]{Percentage cost savings (\%)}
\psfrag{NoStorage}{\scriptsize No storage}
\psfrag{Lyapunov}{\scriptsize Lyapunov}
\psfrag{Greedy}{\scriptsize Greedy}
\psfrag{LowerBoundLegend}{\scriptsize Lower bound}
\psfrag{UpperBoundLegend}{\scriptsize Upper bound}
\includegraphics[width=.33\textwidth]{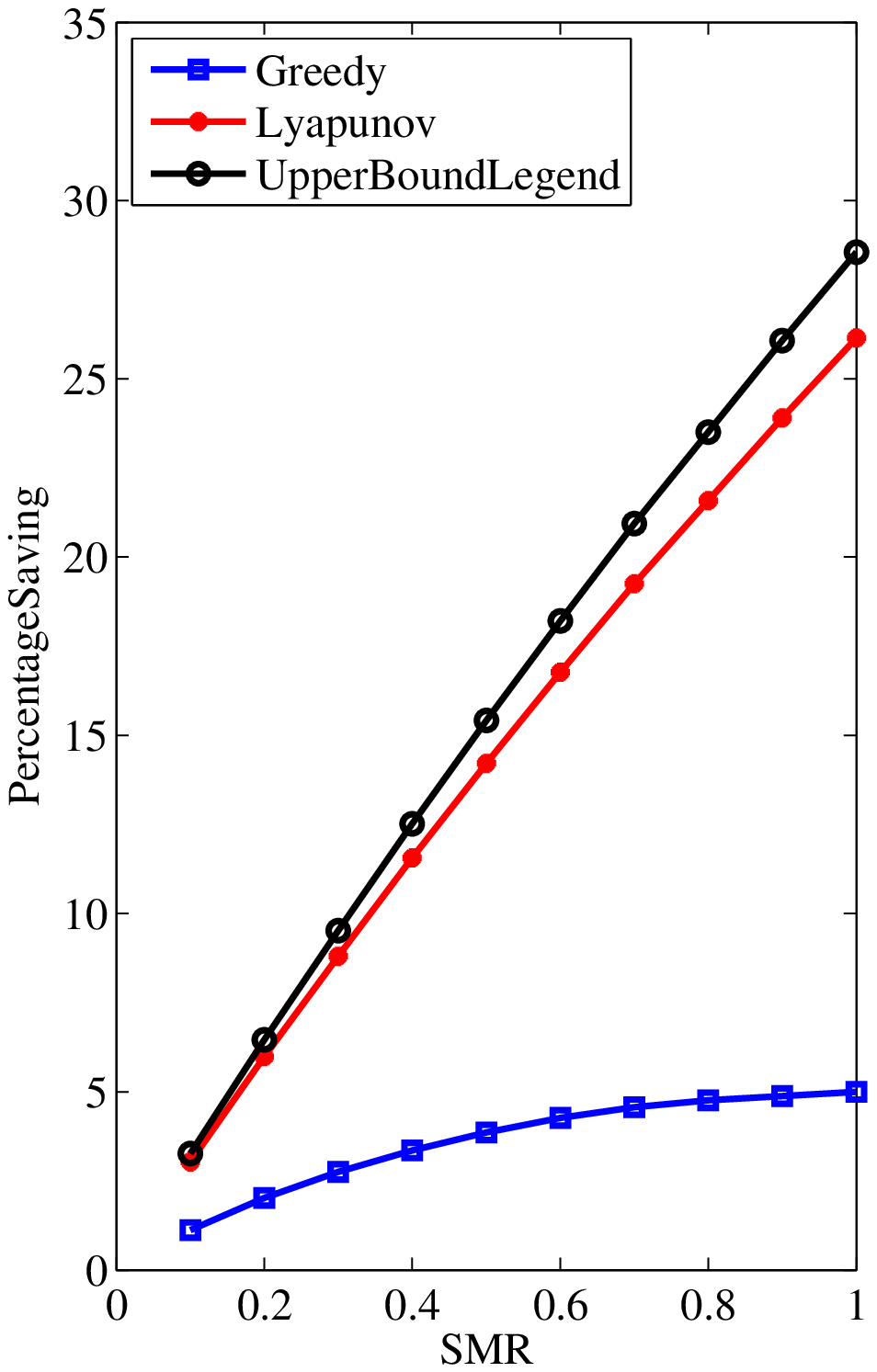}}
\caption{Percentage cost savings of a storage network operated for balancing. }\label{fig:5}
\end{figure}

\section{Conclusions and Future Work}
This work is motivated by the fundamental question of how to {\it optimally} shift energy over space and time to achieve uncertainty reduction and to facilitate renewable integration. To this end, we consider the problem of optimal control of generalized storage networks under uncertainty. The notion of generalized storage is proposed as a dynamic model to capture many forms of storage conveniently. An online control strategy is then proposed to solve the corresponding constrained stochastic control problem, whose performance is analyzed in detail. We prove that the algorithm is near optimal, and its sub-optimality can be bounded by a constant that only depends on the parameters of the storage network and cost functions, and is independent of the realizations of uncertainties. 

Although we have provided analysis for a relatively general setting, potential improvements can be achieved in many directions. (i) Our formulation starts by minimizing the long run  {\it expected average}  cost, and lands on an online algorithm that has {\it robust} performance guarantees in the form of a sub-optimality bound that holds for all uncertainty realizations. Relaxing such requirements may result in approaches that trade {\it risk} with {\it performance}. Better performance guarantees (in terms of smaller sub-optimality) may hold with large probability (instead of with probability one), which leads to, in a sense,  probably approximately correct (PAC) algorithms\cite{strehl2006pac}. (ii) While our online control solves deterministic optimization respecting network constraints, the sub-optimality bound is derived independent of network topology and network parameters such as line capacities. Utilizing such information may lead to a tighter performance bound or a more informed choice of algorithmic parameters. (iii) It can be an advantage or a disadvantage that our online algorithm does not use any statistical information about the uncertain parameters, depending on whether such information is readily available. Observe that our approach in fact can be generalized immediately to settings with additional {\it same-stage variables} which do not affect the (temporal) states of the system.  
Incorporating statistical information and forecast updates may improve the performance of the algorithm, and make the algorithm applicable to other settings where {\it lookahead variables} (such as wind farm contract level for the next stage) are considered together with storage operation. (iv) While the focus of this paper is on energy networks, the algorithm may be applied to other networks since 
our analysis does not rely on properties of the given constraints on network flow. \rev{This also implies that a more accurate AC power model can be used in this study as long as the online optimization can be solved efficiently. Recent advances in tight convex relaxation of AC optimal power flow \cite{BT2013} can be utilized for such purposes. (v) This paper provides a procedure to convert the hard stochastic problem to a sequence of easy deterministic problems which fit into today's grid operational paradigms (especially for transmission grids operated by centralized system operators). 
For the future, the integration of distributed energy resources would require a  decentralized solution to these problems. We note that many methods have been developed for distributed/decentralized {\it deterministic} optimization (\emph{cf.} \cite{BoydAdMM}); incorporating these methods for solving the online problems is an important future direction.}

\section*{Acknowledgments}
This work was supported in part by TomKat Center for Sustainable Energy,  in part by a Satre Family Fellowship, and in part by a Powell Foundation Fellowship. The authors wish to thank Walter Murray for useful comments.

\appendix
\section{Flow Representation of DC Power Flow}\label{sec:DCflow}
To lighten the notation, we drop the dependence of the network parameters on time $t$ as the discussion here applies to any time period.
Let the incidence matrix $\IM \in \real^{m \times n}$ of the network $G(\V,\E)$ be defined by 
\begin{equation}
\IM_{e,i}= 
\begin{cases}
1 & \mbox { if } e = (i, v) \in \E \mbox{ for some }v\in \V, \\
-1 & \mbox{ if } e = (v, i) \in \E \mbox{ for some }v\in \V,\\
0 & \mbox{ otherwise}.
\end{cases}
\end{equation}
Let $\be_{e} >0$ be the admittance of the line $e\in \E$, and $\ben\in \real^m$ be the vector $\{\be_e\}_{e\in \mathcal{E}}$.
Then matrix  $\IM$ is related to the $Y$-bus matrix such that $Y = \IM^\intercal \diag(\ben) \IM$, where $\diag(\ben)$ is the diagonal matrix with $\be_e$, $e\in \mathcal{E}$ on its diagonal. Under DC power flow assumptions, the line flows are linearly related to bus phase angles. Let $\thet_v$ be the phase angle on bus $v$. Then it is easy to see that
\begin{equation}\label{eq:ftheta}
\fn = \diag(\ben) \IM \thetn,
\end{equation}
and the bus flow injection are given by $\IM^\intercal f$. Equation \eqref{eq:ftheta} states that $f$ lives in the range of matrix $H = \diag(\ben) \IM$, \ie, $\fn \in \mathcal{R}(H)$. By the fact that the graph is connected,  $m \ge n-1$. Provided that $\rank(\IM) = n-1$, and $\diag(\ben)$ is full rank, we have $\rank(H) = n-1$. Then $\dim(\mathcal{R}(H)) = n-1$ and thus the dimension of the nullspace of $H^\intercal$ is $m-n+1$, \ie, $\dim (\mathcal{N}(H^\intercal)) = m-n+1$. Let the rows of $\K \in \real^{(m-n+1)\times m}$ be a basis for $\mathcal{N}(H^\intercal)$ (which can be obtained via \eg, singular value decomposition of $H$). Then $\mathcal{R}(K^\intercal) = \mathcal{N}(H^\intercal) \Rightarrow \mathcal{N}(\K) = \mathcal{R}(H)$, and therefore 
\[
\fn = H \thetn \Leftrightarrow \K\fn = 0.
\]

\section{Structural Properties of the Online Optimization}\label{appendix:convexCost}
We consider replacing the $\g(t)$ defined in \eqref{eq:singleBusGeneralCost} with an extended real-valued function 
\begin{equation}
\g(t) = \gc_t(\u(t),\psiaux(t) ),
\end{equation}
where $\psiaux(t)$ is a vector of auxiliary parameters that captures both stochastic parameters and deterministic parameters, and it is supported on a compact set $\C$. 
Observe that this would make our analysis applicable to general cost functions.
Similar to discussions in Section~\ref{sec:proofSingleBus}, we are interested in solving the following optimization in each time period $t$ for $\u(t)$
\begin{subequations}\label{prob:cvx_opt_online}
\begin{align}
\minimize &\la \bs(t) \u(t) + \W \gc_t(\u(t),\psiaux(t) )\label{prob:opt_cvx_1}
\\
\st & \umin \le \u(t) \le \umax,\label{prob:opt_cvx_2}
\end{align}
\end{subequations}
after  observing the realization of $\psiaux(t)$. 

\begin{lemma}[Structural Properties of Single Bus Online Optimization]\label{lem:ana_sol}
For an extended real valued function $\gc_t(\u(t),\psiaux(t) )$, let $\hat \gc^{\psiaux} (\u) \defeq \gc_t(\u,\psiaux)$, where $\psiaux$ is equal to the observed value of $\psiaux(t)$. The following statements hold, regardless of the realizations of $\psiaux(t)$.

\begin{enumerate}
\item if $\la \bs(t) + \W \Dl \hat \gc^{\psiaux}(\u )  \geq 0$, then $\uhat(t) = \umin$;
\item if $\la \bs(t) + \W \Du \hat \gc^{\psiaux}(\u ) \leq 0$, then $\uhat(t) = \umax$.
\end{enumerate}
Here, 
\[
\begin{split}
&\Dl \hat \gc^{\psiaux}(\u ) \defeq \inf\left\{\xi \in \partial \hat \gc^{\psiaux}(\u )\middle|
\begin{array}{l}
 \u\in [\umin,\umax],\,\, \\
\psiaux\in \C
\end{array}
\right\},\\
&\Du \hat \gc^{\psiaux}(\u ) \defeq \sup\left\{\xi \in \partial \hat \gc^{\psiaux}(\u )\middle|
\begin{array}{l}
 \u\in [\umin,\umax],\,\, \\
\psiaux\in \C\end{array}
\right\}.
\end{split}
\]
and $\partial\hat \gc^{\psiaux}(\u )$ is the sub-differential of $\gc^{\psiaux}(\u )$\footnote{We say $\beta$ is a sub-derivative of an extended real-valued function $\gc$ at $\u_0\in\real$ if $\beta\cdot(\u-\u_0) \leq \gc(\u)-\gc(\u_0)$, for any $\u\in\real$, and denote the set of all such $\beta$, namely the sub-differential of $\gc$ at $\u_0$, by  $\partial\gc(\u_0)$.}.
\end{lemma}
\begin{proof}

To show the set of sufficient conditions for $\uhat(t)$ takes $\umax$ (or $\umin$), notice that the condition
\[
\lambda\tilde \b(t) \leq-\W\Du \hat \gc^{\psiaux}(\u ) 
\]
 implies $ \partial J^{\psiaux,\tilde\b}_t (\u)\subseteq(-\infty,0] $ for any $\u(t)$, $\psiaux(t)$, $\tilde\b (t)$. Thus, for every given $\u\in[\umin,\umax]$, if $\beta$ is a constant such that
 \[
 J^{\psiaux,\tilde\b}_t (v)-J^{\psiaux,\tilde\b}_t (\u)\geq \beta\cdot(v-\u),\,\, \forall v\in[\umin,\umax],
  \]
then the sub-differential condition implies that $\beta\leq 0$. Now, by substituting $\u=\umax$ in the above expression, one obtains $\beta\cdot(v-\u)\geq 0$ and $J^{\psiaux,\tilde\b}_t (v)\geq J^{\psiaux,\tilde\b}_t (\umax)$,  for all $v\in[\umin,\umax]$. Therefore, one concludes that $\u(t) = \umax$ attains an optimal solution in problem \eqref{prob:cvx_opt_online}.
 Similarly, the condition 
 \[
 \lambda\tilde \b(t) \geq-\W\Dl \hat \gc^{\psiaux}(\u ) 
 \]
 implies $  \partial J^{\psiaux,\tilde\b}_t (u(t))\subseteq[0,\infty) $. Based on analogous arguments, one concludes that $\u(t) = \umin$ attains an optimal solution in problem \eqref{prob:cvx_opt_online}.
\end{proof}

\section{Proof of Single Bus Results}\label{sec:app:1bus}
\noindent{\bf Proof of Theorem~\ref{thm:feas_single}}
We first validate that the intervals of $\ks$ and $\W$ are non-empty. Note that from Assumption \ref{assume:W_max}, $\Wmax> 0$, thus it remains to show $\ksmax \ge \ksmin$. 
Based on \eqref{eq:W_max}, $W\geq0$, and $\Du\g\geq \Dl\g$, one obtains
\[
\begin{split}
\W(\Du \g - \Dl \g)\leq& \lambda (\bmax-\bmin)-(\umax-(1-\lambda)\bmax)^+\\
&- {((1-\lambda)\bmin-\umin)^+}.
 \end{split}
\]
Re-arranging terms results in 
\[
\begin{split}
 &\left[\!-\W \Dl \g + \pos{\umax \!\!- (1-\la)\bmax}\right]\!\! - \la\bmax\\
 \leq& \left[\!-\W \Du \g - \pos{(1-\la)\bmin \!\!- \umin}\right]\!\! -\la\bmin
 \end{split}
\]
which further implies $\ksmax\geq \ksmin$.

We proceed to show that
\begin{equation}\label{exp:MI}
\bmin\leq  \b(t)\leq \bmax,
\end{equation}
for $t= 1,2,\dots$, when $\u^\star(\Pk{3})$ is implemented. 
The base case holds by assumption.
Let the inductive hypothesis be that (\ref{exp:MI}) holds at time $t$.
The storage level at $t+1$ is  then
$\b(t+1)=\lambda \b(t)+\uhat(t). $
We show  (\ref{exp:MI}) holds at $t+1$ by considering the following three cases.

\noindent{\bf Case 1.}  $- \W \Dl\g\leq\la\bs(t)\leq\la ( \bmax+\ks)$. \\
First, it is easy to verify that the above interval for $\la \bs(t)$ is non-empty using \eqref{eq:ineq_1} and $\ks \ge \ksmin$. Next, based on Lemma \ref{coro:tech_res}, one obtains  $\uhat(t)=\umin\le 0$ in this case. Therefore
\[
\b(t+1) = \la \b(t) + \umin \le  \la \bmax + \umin \le \bmax,
\]
where the last inequality follows from Assumption \ref{assume:feas}. On the other hand, 
\begin{align*}
\b(t+1)& = \la \b(t) + \umin \ge -\W\Dl \g  - \la \ks + \umin \\
\ge & -\W\Dl \g  - \la \ksmax + \umin \\
\ge  &  \pos{(1-\la)\bmin \!\!- \umin}\!+ \!\la \bmin + \umin
\ge    \bmin, 
\end{align*}
where the third line used $\Du \g \ge \Dl \g$. 

\noindent{\bf Case 2.}
 $\la( \bmin+\ks) \le  \la\bs(t)  \le - \W \Du\g$.  \\
The above interval for $\la \bs(t)$ is non-empty by \eqref{eq:ineq_2} and $\ks \le \ksmax$. Lemma \ref{coro:tech_res} implies  $\uhat(t)=\umax \ge 0$ in this case. 
Therefore, using Assumption \ref{assume:feas},
\[
\b(t+1) = \la \b(t) + \umax \ge \la \bmin + \umax \ge \bmin.
\]
On the other hand, 
\begin{align*}
\!\!\b(t+1) &= \la \b(t) + \umax \le -\W\Du g - \la \ks + \umax \\
\le &  -\W\Du g - \la \ksmin + \umax \\
\le &  \!-\! \pos{\umax \!\!- (1-\la)\bmax} \!\! +\! \la \bmax + \umax\!\! 
\le \bmax,
\end{align*}
where the third line again is by $\Du \g \ge \Dl \g$. 

\noindent{\bf Case 3.} $-\W \Du\g<\la\bs(t)<- \W \Dl\g$.  \\
By  $\umin \le \uhat(t)\le\umax$, one obtains
\begin{align*}
\b(t+1) &=  \la \b(t) + \uhat(t) \le \la \b(t) + \umax\\
< & - \W \Dl\g - \la \ks + \umax\\
\le &  - \W \Dl\g - \la \ksmin + \umax\\
\le & \! -\! \pos{\umax \!\!- (1-\la)\bmax}\! +\! \la \bmax + \umax\!\! \le \bmax.
\end{align*}
On the other hand,
\begin{align*}
\b(t+1) &=  \la \b(t) + \uhat(t) \ge \la \b(t) + \umin\\
> & - \W \Dl\g - \la \ks + \umin\\
\ge &  - \W \Dl\g - \la \ksmax + \umax\\
\ge  &  \pos{(1-\la)\bmin \!\!- \umin}\!+ \!\la \bmin + \umin
\ge    \bmin.
\end{align*}

Combining these three cases, and by mathematical induction, we conclude \eqref{exp:MI} holds for all $t = 1,2,\dots$. 
%

\vspace{.2in}
\noindent{\bf Proof of Theorem~\ref{thm:perf_lyap}}
Consider a quadratic Lyapunov function $\L(\b) = \b^2/2$.
Let the corresponding Lyapunov drift be
\[
\Ls( \bs(t)) = \expec\left[ \L(\bs(t+1)) -  \L(\bs(t)) \vert  \bs(t) \right].
\]
Recall that
$
\bs(t+1) = \b(t+1) + \ks = \la \bs(t)   + \u(t) + (1-\la)\ks,
$
and so
\begin{align}
\Ls( \bs(t))
& = \expec\big[ (1/2)(\u(t) + (1-\la)\ks)^2  -  (1/2) (1-\la^2)\bs(t)^2 \nn\\
                &\quad\quad\quad + \la \bs(t) \u(t) + \la (1-\la) \bs(t) \ks \vert  \bs(t) \big]\nn\\
& \le \Mone(\ks)    -  (1/2) (1-\la^2)\bs(t)^2\nn \\
& \quad+\expec\big[ \la \bs(t) \u(t) + \la (1-\la) \bs(t) \ks \vert  \bs(t) \big]\nn\\
& \le\Mone(\ks) + \expec\left[ \la \bs(t) (\u(t)+(1-\la)\ks)\vert \bs(t) \right].\label{eq:1busdrift}
\end{align}
It follows that, with arbitrary storage operation $\u(t)$, 
\begin{align}
& \Ls(\bs(t)) + \W \expec [\g(t)| \bs(t)] \label{eq:ME}\\
\le & \Mone(\ks) + \la (1-\la)\bs(t)\ks+  \expec\big[ \Jkt{3}{t}(\u(t)) | \bs(t)],\nn
\end{align}
%
where it is clear that minimizing the right hand side of the above inequality over $\u(t)$ is equivalent to minimizing the objective of \Pk{3}. Given that $\ustat(t)$, the disturbance-only stationary policy of \Pk{2}  described in Lemma~\ref{lem:stat}, is feasible for \Pk{3}, the above inequality implies
\begin{align}
& \Ls(\bs(t)) + \W \expec [\g(t)| \bs(t), \u(t) = \uhat(t)] \label{eq:ME}\\
\le & \Mone(\ks) + \la (1-\la)\bs(t)\ks+  \expec\big[ \Jkt{3}{t}^\star | \bs(t)] \nn \\
\le & \Mone(\ks) + \la (1-\la)\bs(t)\ks+  \expec\big[ \Jkt{3}{t}(\ustat(t)) | \bs(t)] \nn\\
\stackrel{(a)}{=} & \Mone(\ks) +\la\bs(t)\expec\left[ \ustat(t) +(1-\la)\ks \right]+\W\expec [\g(t)|\ustat(t)]\nn \\
\stackrel{(b)}{\le} & \M(\ks) + \W\expec[\g(t)|\ustat(t)]\stackrel{(c)}{\le}  \M(\ks) + \W \Jk{1}^\star.\nn
\end{align}
Here $(a)$ uses the fact that $\ustat(t)$ is induced by a disturbance-only stationary policy;  $(b)$ follows from inequalities
$ |\bs(t)|\leq \left(\max\left( (\bmax+\ks)^2,(\bmin+\ks)^2\right)\right)^{1/2} $ and
$\left|\expec\left[ \ustat(t)\right]+(1-\la)\ks\right|
\le (1-\la) (\max( (\bmax+\ks)^2,(\bmin+\ks)^2))^{1/2};
$
and $(c)$ used $\expec[\g(t)|\ustat(t)]  = \Jk{2}^\star$ in Lemma~\ref{lem:stat} and $\Jk{2}^\star \le \Jk{1}^\star$. Taking expectation over $\bs(t)$ on both sides gives\begin{align} 
&  \expec\left[ \L(\bs(t+1)) -  \L( \bs(t)) \right] + \W\expec \left[\g(t)|\u(t) = \uhat(t) \right]    \nn\\
  \leq & \M(\ks)+\W \Jk{1}^\star. \label{eq:diff_ly}
\end{align}
Summing expression \eqref{eq:diff_ly} over $t$ from $1$ to $T$, dividing both sides by $\W T$, and taking the limit $T\rightarrow \infty$, we obtain the performance bound in expression \eqref{eq:1busiidperf_bdd}.

\vspace{.2in}
\noindent{\bf Proof of Lemma~\ref{SDP_P3_PO}}
Based on the re-parametrization
$
\None=\Mone(\ks)/\W,\,\, \Ntwo=\Mtwo(\ks)/\W
$
and $\W> 0$, one can easily show that problem \textbf{P3-PO} has a same solution as the following optimization problem:
\begin{subequations}
\begin{align*}
\minimize \quad & \None+\la(1-\la)\Ntwo\\
\st \quad & \ksmin\le  \ks \le \ksmax,0 <  \W  \le \Wmax, \\
&2\None \W\geq  \left(\umin+(1-\la)\ks\right)^2,\\
&2\None \W\geq \left(\umax+(1-\la)\ks\right)^2,\\
&\Ntwo \W\geq  \left(\bmin+\ks\right)^2,\Ntwo \W \geq \left(\bmax+\ks\right)^2.
\end{align*}
\end{subequations}
The proof is completed by applying Schur complement on the last four constraints of the above optimization.

\section{Proof for Networked Systems}\label{sec:proofNetwork}
In this section, we generalize the analysis in Section \ref{sec:proofSingleBus} to the network case. First we have the following assumption on $\d_v(t)$ and $p_v(t,l)$, for $v\in\V$, $l\in\{1,\ldots,L\}$.
\begin{assumption}\label{assume:networkProof}
We assume in this section that at each vertex $v\in \V$, the imbalance process $\{\d_v(t): t\ge 1\}$ and the process $\{\p_v(t,\ell): t \ge 1\}$, $l\in\{1,\ldots,L\}$, follow Assumption \ref{assume:singleBusProof}. 
\end{assumption}
Similar to the single bus system case, we define optimization problem \Pk{1} and \Pk{2} for the networked system, where the state updates on storage levels and the constraints on storage operations are defined on every vertices $v\in\V$. Furthermore, the convex constraint of the network flow is also added to each problem. Also, we define the following vector notations for the networked storage levels, storage operations, shift parameters and shifted storage levels:
\[
\begin{split}
&\bn(t)\defeq\{\b_v(t)\}_{v\in\V},\, \un(t)\defeq\{\u_v(t)\}_{v\in\V},\,\ksn\defeq\{\ks_v(t)\}_{v\in\V},\\
& \tilde\bn(t)=\bn(t)+\ksn.
\end{split}
\]
From the theory of stochastic network optimization \cite{NeelyBook}, the following result holds.
\begin{lemma}
[Optimality of Stationary Disturbance-Only Policies]\label{lem:stat_network}
Under Assumption~\ref{assume:networkProof}
there exists a stationary disturbance-only policy $(\ustatn(t),\fstatn(t))$ satisfying the constraints in \Pk{2}, providing the following guarantees for all $t$:
\begin{align}
  & (1-\la_v)\bmin_v\leq \expec[\ustat_v(t) ] \leq (1-\la_v)\bmax_v, \forall v\in\V\nn \\
  & \expec\left[\sum_{v\in \V}\gstat_v(t)\right]  = \Jk{2}^\star,\nn\end{align}
where the expectation is taken over the randomization of $\d_v(t)$, $\p_v(t,\ell)$, $\ell =1, \dots, L$, $\u_v(t)$, $\f_e(t)$ and 
\begin{equation*}\label{eq:networkGeneralCost_stat}
\begin{split}
&\gstat_v(t) \!=\! \sum_{\ell = 1}^{L_v} \p_v(t, \ell) \Big(\!\alI_v(\ell) \d_v(t) \!-\! \alC_v(\ell) \hC_v\left(\pos{\ustat_v(t)}\!\right) \\
&\!+\! \alD_v(\ell) \hD_v\left(\!\neg{\ustat_v(t)}\!\right)\! +\! \alF_v(\ell)\! \sum_{e\sim v} \!\fstat_e (t)\!+\! \alConst_v(t,\ell)\!\Big)^+. \nn
\end{split}
\end{equation*}
\end{lemma}

Recall the online optimization in expression \eqref{prob:P3_network}. By using $(\uhatn(t),\fhatn(t))$ to denote the minimizers at each time step: 
\[
\fhatn(t)\defeq\{\fhat_e(t)\}_{e\in\E},\, \uhatn(t)\defeq\{\uhat_v(t)\}_{v\in\V},
\]
$\un^\star(\Pk{3})$ and $\fn^\star(\Pk{3})$ to denote the sequence $\{\uhatn(t): t\ge 1\}$, $\{\fhatn(t): t\ge 1\}$ respectively and $\Jkt{3}{t}(\un(t),\fn(t))$ to denote the objective function of \Pk{3} at time period $t$. In the rest of this section, we will analyze feasibility conditions and performance bounds for the Lyapunov optimization algorithm in networked system. Before getting into the feasibility analysis, on each vertex $v\in \V$ we define the following bounds: $\ksmin_v$, $\ksmax_v$ and $\Wmax_v$, by their single bus system counterparts.
\begin{theorem}[Feasibility]
Suppose the initial storage level satisfies $\b_v(1) \in [\bmin_v, \bmax_v]$, for all $v\in\V$, then the storage level sequence $\{\bn(t): t\ge 1 \}$ induced by the sequence of storage operation $(\un^\star(\Pk{3}))$ is feasible with respect to storage level constraints, \ie, $\b_v(t) \in [\bmin_v, \bmax_v]$ for all $t$ and $v\in\V$, provided that
\[
\ksmin_v\le  \ks_v \le \ksmax_v,\,\,0 < \W_v  \le \Wmax_v, 
\]
for all $v\in \V$.
\end{theorem}
\begin{proof}
Based on analogous arguments in Theorem \ref{thm:feas_single}, one can easily validate that the intervals of $\ks_v$ and $\W_v$, $\forall v\in\V$, are non-empty, noting that  
\[
\begin{split}
\W_v(\Du \g_v - \Dl \g_v)\leq& \la_v (\bmax_v-\bmin_v)\!-\!(\umax_v-(1\!-\!\la_v)\bmax_v)^+\\
&- {((1-\la_v)\bmin_v-\umin_v)^+},\,\,\forall v\in\V.
 \end{split}
\]
For the feasibility argument on $\b_v(t)$, $\forall v\in\V$, when $\fn(t)$ is any fixed quantity, one can show $\b_v(t) \in [\bmin_v, \bmax_v]$ by applying Theorem \ref{thm:feas_single} to each vertex $v\in\V$. Since this feasibility result is uniform in variable $\fn(t)$ (both $\Du \g_v$ and $\Dl \g_v$ are independent of $\fn(t)$, and $\fn(t)$ does not explicitly affect the storage level dynamics), the proof is completed by substituting $\fn(t)=\fhatn(t)$, $\forall t$.
\end{proof}
The following theorem provides a performance bound for Lyapunov optimization in networked system. On each vertex $v\in\V$, we also define the following quantities: $\Mone_v(\ks_v)$, $\Mtwo_v(\ks_v)$ by their single bus system counterparts.
\begin{theorem} [Performance]\label{thm:perf_lyap}
The sub-optimality of control sequence $(\un^\star(\Pk{3}),\fn^\star(\Pk{3}))$ is bounded by $\M(\ksn)/\W$, that is
\begin{equation*}\label{eq:networkiidperf_bdd}
\Jk{1}^\star\! \le\!  \Jk{1}(\un^\star(\Pk{3}),\fn^\star(\Pk{3})) \!\le\! \Jk{1}^\star + \sum_{v\in\V}\frac{\M_v(\ks_v)}{\W_v},
\end{equation*}
where
\[
\M_v(\ks_v)=\Mone_v(\ks_v)+\la_v(1-\la_v)\Mtwo_v(\ks_v).
\]

\end{theorem}
\begin{proof}
Consider a quadratic Lyapunov function $\L_v(\b_v) = \b_v^2/2$.
Let the corresponding Lyapunov drift be
\[
\Ls_v( \b_v(t)) = \expec\left[ \L_v(\b_v(t+1)) -  \L_v(\b_v(t)) \vert  \b_v(t) \right].
\]
Similar to the analysis in expression \eqref{eq:1busdrift}, one obtains
\[
\Ls( \bsn(t)) \!\le\!\sum_{v\in\V}\! \Mone(\ksn) +\!\expec\!\left[ \la_v\! \bs_v(t) (\u_v(t)\!+\!(1\!-\!\la_v)\ks_v)\vert \bs_v(t) \right].
\]
It follows that, with arbitrary storage operation $\un(t)$ and network flow $\fn(t)$, 
\begin{align}
& \Ls(\bsn(t)) + \W \expec \left[\sum_{v\in \V} \g_v(t)| \bsn(t)\right] \le  \Mone(\ksn) +\nn\\
& \quad\sum_{v\in\V}\la_v (1-\la_v)\bs_v(t)\ks_v+  \expec\big[ \Jkt{3}{t}(\un(t),\fn(t)) | \bsn(t)],\nn
\end{align}
%
where it is clear that minimizing the right hand side of the above inequality over $(\un(t),\fn(t))$ is equivalent to minimizing the objective of \Pk{3}. Since $(\ustatn(t),\fstatn(t))$, the disturbance-only stationary policy of \Pk{2}  described in Lemma~\ref{lem:stat_network}, is feasible for \Pk{3}, similar to the analysis in expression \eqref{eq:ME}, the above inequality implies
\begin{align}
& \W \expec\! \left[\sum_{v\in \V}\! \g_v(t)| \bsn(t), \uhatn(t),\fhatn(t)\right]\! +\!\Ls(\bsn(t))\nn\\
 \le &  \Mone(\ksn) \!+\!\! \sum_{v\in\V}\la_v\! (1\!-\!\la_v)\bs_v(t)\ks_v\!+\!\expec\big[ \Jkt{3}{t}(\ustatn(t),\!\fstatn(t)) ]\nn\\
 \le &  \M(\ksn) + \W \Jk{1}^\star.\nn
\end{align}
Taking expectation over $\bsn(t)$ on both sides gives\begin{align} 
&\W\expec \left[\sum_{v\in\V}\g_v(t)| \uhatn(t),\fhatn(t)\right]+ \expec\left[ \L(\bsn(t+1)) -  \L( \bsn(t)) \right]\nn\\
  \leq& \M(\ksn)+\W \Jk{1}^\star. \label{eq:diff_ly_network}
\end{align}
Summing expression \eqref{eq:diff_ly_network} over $t$ from $1$ to $T$, dividing both sides by $\W T$, and taking the limit $T\rightarrow \infty$, we obtain the performance bound in expression \eqref{eq:networkiidperf_bdd}.
\end{proof}

The semidefinite program for minimizing the bound is as follows.
\begin{lemma}[Semidefinite Optimization of ${\M(\ksn)}/{\W}$]
Let symmetric positive definite matrices $\Xumin_v$, $\Xumax_v$, $\Xbmin_v$, and $\Xbmax_v$, $v\in\V$, be defined as follows
\begin{equation*}
\!\!\Xudot_v\!\! = \!\!\begin{bmatrix}
\None_v&\!\!\udot_v+(1-\la_v)\ks_v\\
*& 2\W
\end{bmatrix}, \,\,
\Xbdot_v\!\! = \!\!\begin{bmatrix}
\Ntwo_v&\!\!\bdot_v+\ks_v\\
*& \W
\end{bmatrix},\!\!
\end{equation*}
where $(\cdot)$ can be either $\max$ or $\min$, and $\None_v$ and $\Ntwo_v$ are auxiliary variables, for any $v\in\V$. Then the sub-optimality bound $\M(\ksn)/\W$ can be optimized by solving the following semidefinite program
\begin{subequations}\label{prob:sdp_network}
\begin{align}
\emph{\minimize} \quad & \sum_{v\in\V}\None_{v}+\la_v(1-\la_v)\Ntwo_v \\
\emph{\st} \quad & \ksmin_v\le  \ks_v \le \ksmax_v, \label{c1}\\
& 0 <  \W  \le \Wmax_v, \\
& \Xumin_v,\Xumax_v, \Xbmin_v, \Xbmax_v \succeq 0,\label{cend}
\end{align}
\end{subequations}
where constraints \eqref{c1}-\eqref{cend} hold for all $v\in \V$. 
\end{lemma}
\begin{proof}
The proof is similar to the proof of Lemma~\ref{SDP_P3_PO} and is omitted here. 
\end{proof}

\section{Generalization to Markov Cases}\label{appendix_markov_process}
Under Assumption \ref{assume:stoc}, $\thv(t)$ is some deterministic function of the system stochastic state $\omega(t)$, where $\omega(t)$ is a finite state ergodic Markov Chain, supported on $\Omega$. Let $\omega_0\in\Omega$ be the initial state of $\omega(t)$. Since $\omega(t)$ is an ergodic Markov chain, there exists a sequence of finite random return time $T_r$, for $r\in\mathbb N$, such that $\omega(T_r)$ revisits $\omega_0$ for the $r$-th time at time $T_r$. Define $\R(t)$ as the number of visits of $\omega_0$ at time $t$. Specifically, $\R(t)\defeq\max\{r:T_r\leq t\}$. From this sequence of return times, we define the $r-$th epoch as $[T_r,T_{r+1}-1]$ and the length of this epoch is defined as $\DT_r=T_{r+1}-T_r$. From the definition of a return time in a Markov chain, it is obvious that the sequence of $\{\DT_r\}_{r\in\mathbb N}$ is i.i.d.. Let  $\DT$ be a random variable with the (common) time distribution of $\DT_r$, $\forall r$. The positive recurrence assumption implies that $\expec[\DT]<\infty$. We assume that the second moment of $\DT$ is bounded: $\expec\left[(\DT)^2\right]<\infty$ and define the mean return rate of state $w_0$ as $\renewrate = {1}/{\expec[\DT]}$. 

 For the feasibility analysis, the result directly follows from Theorem \ref{thm:feas_single}, as \Pk{3} is a deterministic online optimization problem. Next, we turn to the performance analysis of the Lyapunov optimization algorithm. 
\begin{theorem} [Performance]\label{thm:perf_lyap_renew}
The sub-optimality of storage operation $\u^\star(\Pk{3})$ is bounded by $\M(\ks)/\W$ almost surely, that is
\begin{equation}\label{eq:perf_bdd_renewal}
\Jk{1}^\star \le  \Jk{1}(\u^\star(\Pk{3})) \le \Jk{1}^\star + \Mfour(\ks)/\W
\end{equation}
almost surely, where
\begin{align*}
&\Mfour(\ks)=\Mtwo(\ks)+\renewrate\expec\left [\Mthree(\ks,\DT)\right],\\
&\Mone(\ks) =\! \frac{1}{2}\! \max\left(\! \left(\umin\!+(1-\la)\ks\right)^2\!\!,\left(\umax\!+(1-\la)\ks\right)^2\! \right)\!,\\
&\Mtwo(\ks)= \la(1-\la)\max\left( \left(\bmin+\ks\right)^2,\left(\bmax+\ks\right)^2 \right),\\
&\Mthree(\ks,T)=\Mone(\ks)(2T^2+T),
\end{align*}
\end{theorem}
\begin{proof}
Similar to the case with i.i.d. assumptions, consider a quadratic Lyapunov function $\L(\b) = \b^2/2$. Let the corresponding Lyapunov drift be
\[
\Ls( \bs(t)) = \expec\left[ \L(\bs(t+1)) -  \L(\bs(t)) \vert  \bs(t) \right].
\]
Based on the analysis in expression (\ref{eq:1busdrift}), one obtains
\[
\Ls( \bs(t)) \le\Mone(\ks) + \expec\left[ \la \bs(t) (\u(t)+(1-\la)\ks)\vert \bs(t) \right].
\]
By substituting $t=T_r$, and by a telescoping sum, it follows that, with arbitrary storage operation $\u(t)$, 
\begin{equation*}
\begin{split}
&\expec\left[L(\tilde \b(T_{r+1})) \!-\! L(\tilde \b(T_r))\!+\! \W\sum_{\tau=T_r}^{T_{r+1}-1}\!\g(\tau)\vert \tilde \b(T_r)\!\right] \\
 \leq& \expec\left[ \sum_{\tau=T_r}^{T_{r+1}-1}\la \bs(\tau) \u(\tau)\! +\! \W\g(\tau)\vert \tilde \b(T_r)\right]\!+\!\\
 &\Mone(\ks)\expec[\DT_r\vert \tilde \b(T_r)]\!+\!\expec\left[\sum_{\tau=T_r}^{T_{r+1}-1}\la \bs(\tau)(1-\la)\ks)\vert\bs(T_r)\right].
\end{split}
\end{equation*}
It is clear that minimizing the right hand side of the above expression over $\u(\tau)$ is equivalent to minimizing the objective of \Pk{3}. Given that $\ustat(\tau)$, the disturbance-only stationary policy of \Pk{2}  described in Lemma~\ref{lem:stat}, is feasible for \Pk{3}. Thus, this expression implies
\begin{align}
&  \expec \left[\sum_{\tau=T_r}^{T_{r+1}-1}\Ls( \bs(\tau))+\W\g(\tau)| \bs(T_r), \u(\tau) = \uhat(\tau),\,\forall \tau\right]\nn\\
\le &  \expec\left[\sum_{\tau=T_r}^{T_{r+1}-1}\Mone(\ks) + \la (1-\la)\bs(\tau)\ks+\Jkt{3}{\tau}\vert\bs(T_r)\right]\nn\\
\le &  \expec\left[\!\sum_{\tau=T_r}^{T_{r+1}-1} \Jkt{3}{\tau}(\ustat(\tau)) | \bs(T_r)\right]\nn \\
&+\expec\left[\sum_{\tau=T_r}^{T_{r+1}-1}\Mone(\ks) + \la (1-\la)\bs(\tau)\ks\vert\bs(T_r)\right]\nn\\
=&\W\expec\left [\sum_{\tau=T_r}^{T_{r+1}-1}\gstat(\tau)\right]\nn\\
&+\expec\left[\sum_{\tau=T_r}^{T_{r+1}-1}\!\! \!\Mone(\ks) +\la\bs(\tau) (\ustat(\tau) +(1-\la)\ks)\vert\bs(T_r)\! \right]\label{eq:opt_renewal_1}
\end{align}
where the first inequality follows from the definition of $\Jkt{3}{\tau}$ and the first equality follows from the facts that $\{\ustat(\tau)\}_{T_r}^{T_{r+1}-1}$ is a sequence of disturbance-only control polices and each epoch: $[T_r,T_{r+1}-1]$, $\forall r$, is independent and identically distributed. Based on the definitions of $\tilde \b(\tau)$ and $\Mone(\ks)$, one obtains $|(\ustat(\tau)+(1-\la)\ks) |\leq\sqrt{2\Mone(\ks)}$, $\forall \tau$ and
\[
\begin{split}
\sum_{\tau=T_r}^{T_{r+1}-1}\!\!\!|( \bs(\tau)-&\bs(T_r))|\!
\leq \!\!\!\!\sum_{\tau=T_r}^{T_{r+1}-1}\!\left|\sum_{j=T_r}^{\tau-1}\!\!\la^{\tau-j-1} (\u(j)\!+\!(1\!-\!\la)\ks)\right|\\
\leq&\sum_{\tau=T_r}^{T_{r+1}-1}\sum_{j=T_r}^{\tau-1} |(u(j)+(1-\la)\ks)|\\
\leq&\sqrt{2\Mone(\ks)} \sum_{\tau=T_r}^{T_{r+1}-1}\DT_r \leq \sqrt{2\Mone(\ks)}\DT_r^2.
\end{split}
\]
By the Renewal Cost Theorem (Theorem 3.6.1, \cite{ross_stochastic}), with $\{\DT_r,\sum_{\tau=T_r}^{T_{r+1}-1}\gstat(\tau)\}$ and $\{\DT_r,\sum_{\tau=T_r}^{T_{r+1}-1}\ustat(\tau)\}$, $\forall r$, forming two i.i.d. processes, one obtains 
\[
\begin{split}
\renewrate\expec\left[ \sum_{\tau=T_r}^{T_{r+1}-1}\gstat(\tau)\right]\!=\!&\lim_{t\to \infty} \frac{1}{t} \expec \left[ \sum_{r=0}^{\R(t)-1}\sum_{\tau=T_r}^{T_{r+1}-1}\gstat(\tau)\right]\\
\!=\!&\lim_{t\to \infty} \frac{1}{t} \expec \left[\sum_{\tau=1}^{t} \gstat(\tau)\right],\\
\renewrate\expec\left[ \sum_{\tau=T_r}^{T_{r+1}-1}\ustat(\tau)\right]\!=\!&\lim_{t\to \infty} \frac{1}{t} \expec \left[ \sum_{r=0}^{\R(t)-1}\sum_{\tau=T_r}^{T_{r+1}-1}\ustat(\tau)\right]\\
\!=\!&\lim_{t\to \infty} \frac{1}{t} \expec \left[\sum_{\tau=1}^{t} \ustat(\tau)\right].
\end{split}
\]
Since the epoch duration $\DT_r$ is independent of the shifted queueing state $\tilde{\b}(T_r)$, based on the definition of $\Mthree(\ks,\DT_r)$, one obtains 
\[
\expec\left [\Mthree(\ks,\DT_r)\vert \tilde \b(T_r)\right]= \expec\left [\Mthree(\ks,\DT_r)\right].
\]
Thus, by combining all these arguments, expression (\ref{eq:opt_renewal_1}) becomes
\begin{equation}
 \begin{split}
& \expec \left[\sum_{\tau=T_r}^{T_{r+1}-1}\Ls( \bs(\tau))+\W\g(\tau)| \bs(T_r), \u(\tau) = \uhat(\tau),\,\forall \tau\right]\nn\\
\le & \W\expec[\DT_r]\lim_{t\to \infty} \frac{1}{t} \sum_{\tau=1}^t\expec \left[ \gstat(\tau)\right]+\expec\left [2\Mone(\ks)\DT_r^2\right] \\
&+\la\tilde \b(T_r)\expec[\DT_r]\left(\lim_{t\to \infty} \frac{1}{t} \expec \left[\sum_{\tau=1}^t  \ustat(\tau)\right]+(1-\la)\ks\right) 
\end{split}
 \end{equation}
Furthermore, recall that epoch duration $\DT_r$ is independent and identically distributed. Therefore, there exists a random variable $\DT$ with the common inter-arrival time distribution, such that 
\[
\expec\left[\DT^2_r\right]= \expec\left[\left(\DT\right)^2\right],\,\,\expec[\DT_r]= \expec\left[\DT\right],\,\, \forall r.
\]
Note that $\lim_{t\to \infty} \frac{1}{t} \sum_{\tau=1}^t\expec \left[ \gstat(\tau)\right]=\Jk{2}^\star\leq \Jk{1}^\star$. By taking expectation in the above expression with respect to $\tilde \b(T_r)$, one obtains
\[
\begin{split}
&\expec\left[\sum_{\tau=T_r}^{T_{r+1}-1}\Ls( \bs(\tau))+\W\g(\tau)| \u(\tau) = \uhat(\tau),\,\forall \tau\right] \\
\leq&\expec\left [\Mthree(\ks,\DT_r)\right] +\Mtwo(\ks)\expec\left [\DT_r\right]  + \W \expec[\DT_r] \Jk{1}^\star.
\end{split}
\]
By a telescoping sum over $r=0,\ldots, \R(t)-1$, and notice that 
\[
T_0=1,\,\, T_{r+1}=T_{r}+\DT_r,\,\,\forall r\geq 0,
\]
the above expression implies 
\begin{equation}
\begin{split}
&\expec\left[\sum_{\tau=1}^{T_{\R(t)}}\Ls( \bs(\tau))+\W\g(\tau)| \u(\tau) = \uhat(\tau),\,\forall \tau\right] \leq \\
&    \R(t)\left(\expec\left [\Mthree(\ks,\DT)\right] +(\Mtwo(\ks) +\W \Jk{1}^\star) \expec\left [\DT\right] \right).
\end{split}
 \end{equation}
 Dividing both sides the above expression by $\R(t) \expec \left[\DT\right]$, and recalling the definition of $\Mfour(\ks)$, the above expression becomes
 \begin{equation}
 \begin{split}\label{eq:renewal_2}
& \frac{\W }{\R(t) \expec[\DT]}\expec\left[ \sum_{\tau=1}^{T_{\R(t)}}\g(\tau)|  \uhat(\tau),\,\,\forall\tau\right]+\\
&\frac{1}{\R(t) \expec[\DT]}\expec\left[L(\tilde \b(T_{\R(t)})) - L(\tilde \b(0))\right]\\
\leq &  \Mfour(\ks)+ \W \Jk{1}^\star.\\
\end{split}
 \end{equation}
From the assumption of Lyapunov function, we know that for any $\R(t)\in\mathbb N$, $L(\tilde \b(T_{\R(t)})) \geq 0$ and $0\leq L(\tilde \b(0))<\infty$. Also, recall $t_r\rightarrow \infty$ as $r\rightarrow\infty$ and $\R(t)=\max\{r:T_r\leq t\}$. Then, as $t\rightarrow\infty$, we get $\R(t)\rightarrow\infty$. This implies that 
\[
\lim_{t\rightarrow\infty}\!\frac{\expec\left[L(\tilde \b(T_{\R(t)})) - L(\tilde \b(0)) \right]}{\R(t) \expec[\DT]}\!=\!\!\lim_{t\rightarrow\infty}\!\frac{\expec\left[L(\tilde \b(T_{\R(t)})) \right]}{\R(t) \expec[\DT]}
\]
which is a positive quantity. By taking the limit on $t\rightarrow \infty$, and dividing both sides by $\W\in(0,\Wmax]$, expression \eqref{eq:renewal_2} implies
 \begin{equation}
 \begin{split}
&\lim_{t\rightarrow \infty} \frac{
\expec\left[ \sum_{\tau=1}^{t}\g(\tau)-\sum_{\tau=T_{\R(t)}+1}^{t}\g(\tau)| \uhat(\tau),\,\,\forall\tau\right] }{t}\cdot\\
&\frac{t}{\R(t) \expec[\DT]}\leq  \frac{\Mfour(\ks)}{\W}+  \Jk{1}^\star.
\end{split}
 \end{equation}
Now, since $T_{\R(t)}+1\leq t\leq T_{\R(t)+1}$, and $T_{\R(t)+1}-T_{\R(t)}=\DT_{\R(t)}$, by letting 
\[
\gmax = \max\left\{|\g(t)|\middle|
\begin{array}{l}
 \u(t)\in [\umin,\umax],\,\, \\
\p(t, \ell)\in [\pmin(\ell), \pmax(\ell)],\, \forall \ell,\\
\d(t)\in [\dmin, \dmax],\\
\fn(t)\in\F
\end{array}
\right\}
\]
and noting that $\expec\left[\DT\right],\expec\left[(\DT)^2\right]<\infty$ and $|\g(t)|\leq \gmax$ for each time slot $t$, one easily obtains
\[
\begin{split}
0\!\leq\!&\lim_{t\rightarrow \infty}\!\frac{1}{\R(t) \expec[\DT]}\left|\expec\!\left[ \sum_{\tau=T_{\R(t)}+1}^{t}\!\g(\tau)| \uhat(\tau),\,\forall\tau\!\right]\right|\\
\leq& \lim_{t\rightarrow \infty}\frac{\expec[\DT_{\R(t)}] \gmax}{\R(t) \expec[\DT]}=\lim_{t\rightarrow \infty} \frac{\gmax\expec[\DT]}{\R(t)\expec[\DT] }=0.
\end{split}
\]
The last equality is due to the fact that $\R(t)\rightarrow\infty$ when $t\rightarrow\infty$. Next, recall from the Elementary Renewal Theory (Theorem 3.3.4, \cite{ross_stochastic}) that
\[
\lim_{t\rightarrow\infty} \frac{t}{\R(t)\expec[\DT]}=1\,\, \text{almost surely}.
\]
By combining all previous arguments, one further obtains the following expression:
\begin{equation*}
\lim_{t\rightarrow \infty} \frac{1}{t}\expec\left[ \sum_{\tau=1}^{t-1}\g(\tau)| \u(\tau) = \uhat(\tau),\,\,\forall\tau\right] 
\leq  \frac{\Mfour(\ks)}{W}+   \Jk{1}^\star
 \end{equation*}
almost surely. Finally, from the definition of $\Jk{1}(\u^\star(\Pk{3}))$, the above inequality implies expression \eqref{eq:perf_bdd_renewal}.
 \end{proof}

\bibliography{jqin}

\begin{thebibliography}{10}

\bibitem{QCYR:acm}
J.~Qin, Y.~Chow, J.~Yang, and R.~Rajagopal.
\newblock {Modeling and Online Control of Generalized Energy Storage Networks}.
\newblock In {\em Proc. of the 5th International Conference on Future Energy
  Systems (ACM e-Energy '14)}. ACM, June 2014.

\bibitem{OMGarxiv}
J.~{Qin}, Y.~{Chow}, J.~{Yang}, and R.~{Rajagopal}.
\newblock {Online Modified Greedy Algorithm for Storage Control under
  Uncertainty}.
\newblock {\em {ArXiv} e-prints}, 2014.

\bibitem{TSGarxiv}
J.~{Qin}, Y.~{Chow}, J.~{Yang}, and R.~{Rajagopal}.
\newblock {Distributed Online Modified Greedy Algorithm for Networked Storage
  Operation under Uncertainty}.
\newblock {\em {ArXiv} e-prints}, 2014.

\bibitem{NRELWest2010}
{National Renewable Energy Laboratory}.
\newblock {Western Wind and Solar Integration Study}, 2010.

\bibitem{Denholm2010}
{National Renewable Energy Laboratory}.
\newblock {The Role of Energy Storage with Renewable Electricity Generation},
  2010.

\bibitem{lindley2010naturenews}
D.~Lindley.
\newblock {Smart Grids: The Energy Storage Problem}.
\newblock {\em Nature}, 463(7277):18, 2010.

\bibitem{thermalStor1993}
B.~Daryanian and R.~E. Bohn.
\newblock {Sizing of Electric Thermal Storage under Real Time Pricing}.
\newblock {\em {IEEE} Transactions on Power Systems}, 8(1):35--43, 1993.

\bibitem{ObRACC2013}
G.~O'Brien and R.~Rajagopal.
\newblock {A Method for Automatically Scheduling Notified Deferrable Loads}.
\newblock In {\em Proc. of American Control Conference}, pages 5080--5085,
  2013.

\bibitem{HaoSanandajiPoollaVincent2013}
H.~Hao, B.~M. Sanandaji, K.~Poolla, and T.~L. Vincent.
\newblock {Aggregate Flexibility of Thermostatically Controlled Loads}.
\newblock {\em IEEE Transactions on Power Systems}, submitted.

\bibitem{QRsimpleStorPes2012}
J.~Qin, R.~Sevlian, D.~Varodayan, and R.~Rajagopal.
\newblock {Optimal Electric Energy Storage Operation}.
\newblock In {\em Proc. of IEEE Power and Energy Society General Meeting},
  pages 1--6, 2012.

\bibitem{MITrampStor}
A.~{Faghih}, M.~{Roozbehani}, and M.~A. {Dahleh}.
\newblock {On the Economic Value and Price-Responsiveness of Ramp-Constrained
  Storage}.
\newblock {\em {ArXiv} e-prints}, 2012.

\bibitem{SuEGTPS}
H.~I. Su and A.~El~Gamal.
\newblock {Modeling and Analysis of the Role of Energy Storage for Renewable
  Integration: Power Balancing}.
\newblock {\em IEEE Transactions on Power Systems}, 28(4):4109--4117, 2013.

\bibitem{RLDSACC}
J.~Qin, H.~I. Su, and R.~Rajagopal.
\newblock {Storage in Risk Limiting Dispatch: Control and Approximation}.
\newblock In {\em Proc. of American Control Conference (ACC)}, pages
  4202--4208, 2013.

\bibitem{2012arXiv1212.0272Q}
J.~{Qin}, H.-I {Su}, and R.~{Rajagopal}.
\newblock {Risk Limiting Dispatch with Fast Ramping Storage}.
\newblock {\em ArXiv e-prints}, December 2012.

\bibitem{6672872}
J.~Qin and R.~Rajagopal.
\newblock {Dynamic Programming Solution to Distributed Storage Operation and
  Design}.
\newblock In {\em Proc. of Power and Energy Society General Meeting (PES)},
  pages 1--5, July 2013.

\bibitem{BitarRACC_colocated}
E.~Bitar, R.~Rajagopal, P.~Khargonekar, and K.~Poolla.
\newblock {The Role of Co-Located Storage for Wind Power Producers in
  Conventional Electricity Markets}.
\newblock In {\em Proc. of American Control Conference (ACC)}, pages
  3886--3891, 2011.

\bibitem{Powell}
J.~H. Kim and W.~B. Powell.
\newblock {Optimal Energy Commitments with Storage and Intermittent Supply}.
\newblock {\em Operations Research}, 59(6):1347--1360, 2011.

\bibitem{IBMload}
P.~M. {van de Ven}, N.~{Hegde}, L.~{Massoulie}, and T.~{Salonidis}.
\newblock {Optimal Control of End-User Energy Storage}.
\newblock {\em {ArXiv} e-prints}, 2012.

\bibitem{DataCenter}
R.~Urgaonkar, B.~Urgaonkar, M.~J. Neely, and A.~Sivasubramaniam.
\newblock {Optimal Power Cost Management Using Stored Energy in Data Centers}.
\newblock In {\em Proc. of the {ACM} SIGMETRICS Joint International Conference
  on Measurement and Modeling of Computer Systems}, SIGMETRICS '11, pages
  221--232, 2011.

\bibitem{StorDRLongbo}
L.~Huang, J.~Walrand, and K.~Ramchandran.
\newblock {Optimal Demand Response with Energy Storage Management}.
\newblock In {\em Proc. of IEEE Third International Conference on Smart Grid
  Communications}, pages 61--66, 2012.

\bibitem{XieEtAlWindStorMPC}
L.~Xie, Y.~Gu, A.~Eskandari, and M.~Ehsani.
\newblock {Fast MPC-Based Coordination of Wind Power and Battery Energy Storage
  Systems}.
\newblock {\em Journal of Energy Engineering}, 138(2):43--53, 2012.

\bibitem{NRELStorValue2013}
{National Renewable Energy Laboratory}.
\newblock {The Value of Energy Storage for Grid Applications}, 2013.

\bibitem{NeelyBook}
M.~J. Neely.
\newblock {Stochastic Network Optimization with Application to Communication
  and Queueing Systems}.
\newblock {\em Synthesis Lectures on Communication Networks}, 3(1):1--211,
  2010.

\bibitem{MS2006}
P.~Mokrian and M.~Stephen.
\newblock {A Stochastic Programming Framework for the Valuation of Electricity
  Storage}.
\newblock In {\em Proc. of 26th USAEE/IAEE North American Conference}, pages
  24--27, 2006.

\bibitem{Stott09}
B.~Stott, J.~Jardim, and O.~Alsac.
\newblock {DC Power Flow Revisited}.
\newblock {\em {IEEE} Transactions on Power Systems}, 2009.

\bibitem{strehl2006pac}
A.~Strehl, L.~Li, E.~Wiewiora, J.~Langford, and M.~Littman.
\newblock {PAC Model-Free Reinforcement Learning}.
\newblock In {\em {Proc. of the 23rd International Conference on Machine
  Learning}}, pages 881--888. ACM, 2006.

\bibitem{BT2013}
B.~Zhang and D.~Tse.
\newblock {Geometry of Feasible Injection Region of Power Networks}.
\newblock {\em {IEEE} Transactions on Power Systems}, 28(2):788--797, 2013.

\bibitem{BoydAdMM}
S.~Boyd, N.~Parikh, E.~Chu, B.~Peleato, and J.~Eckstein.
\newblock {Distributed Optimization and Statistical Learning Via the
  Alternating Direction Method of Multipliers}.
\newblock {\em Foundations and Trends{\textregistered} in Machine Learning},
  3(1):1--122, 2011.

\bibitem{ross_stochastic}
S.~Ross.
\newblock {\em {Stochastic Processes}}.
\newblock Wiley Series in Probability and Statistics, 1995.

\end{thebibliography}
\bibliographystyle{unsrt}

\end{document}